\documentclass[10pt, a4paper,reqno]{amsart}

\usepackage{color} \definecolor{bleu_sombre}{rgb}{0,0,0.6}  \definecolor{rouge_sombre}{rgb}{0.8,0,0}\definecolor{vert_sombre}{rgb}{0,0.6,0}
\usepackage[plainpages=false,colorlinks,linkcolor=bleu_sombre,citecolor=rouge_sombre,urlcolor=vert_sombre,breaklinks]{hyperref}

\usepackage[english]{babel}

\usepackage{amsmath,amssymb,amsthm,amsfonts,url,color,enumerate,dsfont}

\theoremstyle{plain}
\newtheorem{theorem}{{Theorem}}[section] 
\newtheorem*{theorem*}{{Theorem}}
\newtheorem{proposition}[theorem]{Proposition}
\newtheorem*{proposition*}{Proposition}
\newtheorem{corollary}[theorem]{Corollary}
\newtheorem*{corollary*}{Corollary}
\newtheorem{lemma}[theorem]{Lemma}
\newtheorem*{lemma*}{Lemma}

\theoremstyle{definition}
\newtheorem{definition}[theorem]{Definition}
\newtheorem*{definition*}{Definition}

\theoremstyle{remark}
\newtheorem{remark}[theorem]{Remark}
\newtheorem{example}[theorem]{Example}

\makeatletter

\@addtoreset{equation}{section}  
\makeatother

\renewcommand{\leq}{\leqslant}	\renewcommand{\geq}{\geqslant}
\renewcommand{\bar}[1]{\overline{#1}}
\renewcommand\over[2]{{\,\buildrel #1\over#2\,}}

\newcommand{\inv}{^{-1}}

\newcommand {\limt}[2]{\xrightarrow[#1 \to #2]{}}

\newcommand{\abs}[1]{\left\vert #1\right\vert}        
\newcommand{\nr}[1]{\left\Vert #1\right\Vert}         
\newcommand{\innp}[2]{\left< #1 , #2 \right>}         

\newcommand{\Dom}{\Dc}			

\newcommand{\pppg}[1] {\left< #1 \right>} 	

\newcommand{\singl}[1]{\left\{ #1 \right\}}		
\newcommand{\Ii}[2] {\{#1,\dots,#2\}}         
\newcommand{\R}{\mathbb{R}}		\newcommand{\C}{\mathbb{C}}
\newcommand{\N}{\mathbb{N}}

\newcommand{\1}[1]{\mathds 1 _{#1}}

\newcommand{\st}{\,:\,}					

\newcommand{\seq}[2]{\left({#1}_{#2}\right)_{#2 \in\N}} 

\newcommand{\restr}[2]{\left.#1\right|_{#2}}         
\renewcommand{\Re}{\mathop{\rm{Re}}\nolimits}        
\renewcommand{\Im}{\mathop{\rm{Im}}\nolimits}        
\DeclareMathOperator{\Ran}{Ran}	

\DeclareMathOperator{\Id}{Id}                        
 
\DeclareMathOperator{\supp}{supp}                    

\renewcommand{\a}{\alpha}\renewcommand{\b}{\beta}\newcommand{\g}{\gamma}\newcommand{\G}{\Gamma}\renewcommand{\d}{\delta}\newcommand{\D}{\Delta}\newcommand{\e}{\varepsilon}\newcommand{\z}{\zeta} \renewcommand{\th}{\theta}\newcommand{\Th}{\Theta}\renewcommand{\l}{\lambda}\newcommand{\m}{\mu}\newcommand{\n}{\nu}\newcommand{\x}{\xi}\newcommand{\s}{\sigma}\renewcommand{\t}{\tau}\newcommand{\f}{\varphi}\newcommand{\vf}{\phi}\newcommand{\h}{\chi}\newcommand{\p}{\psi}\renewcommand{\o}{\omega}\renewcommand{\O}{\Omega}

\newcommand{\Cc}{{\mathcal C}}\newcommand{\Dc}{{\mathcal D}}\newcommand{\Ec}{{\mathcal E}}\newcommand{\Hc}{{\mathcal H}}\newcommand{\Kc}{{\mathcal K}}\newcommand{\Lc}{{\mathcal L}}\newcommand{\Rc}{{\mathcal R}}\newcommand{\Sc}{{\mathcal S}}\newcommand{\Tc}{{\mathcal T}}

\newcommand{\ad}{{\rm{ad}}}

\newcounter{stepproof}
\newcommand{\stepp}{\stepcounter{stepproof} \noindent {\bf $\bullet$}\quad }

\newcommand{\detail}[1]
{
}

\title{Mourre's method for a dissipative form perturbation}
\author{Julien Royer}
\address{Institut de Math\'ematiques de Toulouse \\ 118, route de Narbonne \\ 31062 Toulouse C\'edex 09 \\ France}
\email{julien.royer@math.univ-toulouse.fr}

\begin{document}

\newcommand{\Ha}{H_a}
\newcommand{\Ho}{H_0}
\newcommand{\tHo}{\tilde H_0}
\newcommand{\tH}{\tilde H}
\newcommand{\Gze}{G_z(\e)}

\begin{abstract}
We prove uniform resolvent estimates for an abstract operator given by a dissipative perturbation of a self-adjoint operator in the sense of forms. For this we adapt the commutator method of Mourre. We also obtain the limiting absorption principle and uniform estimates for the derivatives of the resolvent. This abstract work is motivated by the Schr\"odinger and wave equations on a wave guide with dissipation at the boundary. 
\end{abstract}

\maketitle

\section{Introduction}

The purpose of this paper is to prove some uniform resolvent estimates and the limiting absorption principle for a dissipative operator obtained by a form-perturbation of a self-adjoint operator. For this we prove a suitable version of the commutators method of Mourre.\\

Given a self-adjoint operator $H_0$ on a Hilbert space $\Hc$, the purpose of the Mourre method (see \cite{mourre81}) is to prove uniform estimates for the weighted resolvent 
\begin{equation} \label{intro-res-estim}
\pppg A^{-\d} (H_0 -z)\inv \pppg A^{-\d}
\end{equation}
when $z$ is close to the real axis. Here $A$ is a (self-adjoint) conjugate opearator and $\d > \frac 12$. The main assumption on $A$ concerns its commutator with $H_0$:
\begin{equation} \label{intro-hyp-mourre}
\1 J (H_0) [H_0 ,iA] \1 J(H_0) \geq \a \1 J (H_0), \quad \text{for some } \a > 0.
\end{equation}
Here $J$ is an open interval of $\R$ and $\1 J$ is the characteristic function of $J$. With this (and other) assumption(s), the operator \eqref{intro-res-estim} appears to be uniformly bounded for $\Re(z)$ in a compact subset of $J$ and $\Im(z) \neq 0$. In addition to these resolvent estimates, the method gives the limiting absorption principle: the operator \eqref{intro-res-estim} is not only uniformly bounded for $z$ close to the real axis, but for $\Re(z) \in J$ it has a limit when $\pm \Im (z) \searrow 0$.\\

The original motivation for this theory is to prove the absence of singular spectrum for $H_0$ in $J$. This is an important question in scattering theory. 

Compared to previous commutators methods (see for instance \cite{putnam, lavine69, lavine71, lavine73}), we see that the assumption \eqref{intro-hyp-mourre} on the commutator is spectrally localized with respect to $H_0$. This result proved to be very efficient for difficult problems such as the $N$-body problem (see for instance \cite{perryss81, derezinski-gerard, hunzikers00a} and references therein).\\

There are so many generalisations of the original result that we cannot mention them all, so we refer to \cite{amrein} for a general overview of the subject. See also \cite{cycon}.\\

In this paper we focus on dissipative operators. In \cite{art-mourre}, we generalized the result of \cite{mourre81} for a dissipative operator $H = H_0 - i V$, where $V \geq 0$ is relatively bounded with respect to $H_0$. In this case we cannot localize spectrally with respect to the non-selfadjoint operator $H$, but it turned out that we can obtain a similar result using the spectral projections of the self-adjoint part $H_0$. It is even possible to use the dissipative part to weaken the assumption:
\begin{equation} \label{intro-hyp-mourre-diss}
\1 J (H_0) \big( [H_0 ,iA] + \b V \big) \1 J(H_0) \geq \a \1 J (H_0), \quad \a > 0, \quad \b \geq 0.
\end{equation}
Notice that for a general maximal dissipative operator we only know that the spectrum is included in the lower half-plane $\singl{\Im(z) \leq 0}$ and the estimates for the weighted resolvent \eqref{intro-res-estim} (with $H_0$ replaced by $H = H_0 -iV$) are only available for $\Im(z) > 0$.
Then in \cite{boucletr14} we adapted to this setting the results of \cite{jensenmp84, jensen85} about the derivatives of the resolvent.  We also mention \cite{boussaidg10} for a closely related context.

The present work is motivated by the dissipative wave guide. If we consider a Schr\"odinger operator on a domain with dissipation at the boundary, we obtain a dissipative operator $H$ which cannot be written as $H_0 - iV$ for $H_0$ and $V \geq 0$ as in \cite{art-mourre}. However the quadratic form $q$ associated to $H$ can be written as $q = q_0 - i q_\Th$ where $q_0$ is the quadratic form corresponding to a self-adjoint operator and $q_\Th$ is a non-negative quadratic form relatively bounded with respect to $q_0$. This example will be discussed with more details in Section \ref{sec-guide}. Our main purpose in this paper is to prove uniform estimates for the resolvent of this kind of operators, as well as estimates for the derivatives of the resolvent and the limiting absorption principle. A closely related result has been proved in \cite{amreinbg} for self-adjoint operators. Moreover the Mourre method has already been used for wave guides (in a self-ajdoint context) in \cite{krejcirikt04}.\\

Compared to the self-adjoint analog, the first motivation for proving a dissipative Mourre theorem is not to obtain results on the absolutely continuous spectrum and the corresponding absolutely continuous subspace. Indeed, they are \emph{a priori} only defined for self-adjoint operators. However, an absolutely continuous subspace corresponding to a maximal dissipative operator $H$ on $\Hc$ has been defined in \cite{davies78} as the closure in $\Hc$ of 
\[
\singl{ \f \in \Hc \st \exists C_\f \geq 0, \forall \p \in \Hc,  \int_0^{+\infty} \abs {\innp{ e^{-itH} \f}\p}^2 \, dt \leq C_\f \nr \p_\Hc^2}. 
\]
This definition coincide with the usual one for a self-adjoint operator. Notice that there are other generalizations for the notion of absolutely continuous subspace in the litterature (see for instance \cite{nagyf, romanov04, romanov06, ryzhov97, ryzhov97b, ryzhov98}). We prove in this paper that the uniform resolvent estimates given by the Mourre theory give results on the the absolutely continuous subspace in the sense of Davies. For this we will use the dissipative generalization of the theory of relatively smooth operators in the sense of Kato.\\

This paper is organized as follows. In Section \ref{sec-def} we give precise definitions for the dissipative operator $H$ which we consider and the corresponding conjugate operator $A$. Then in Section \ref{sec-guide} we describe the applications which motivated this abstract work: the Schr\"odinger operator on a wave-guide or on a half-space with dissipation at the boundary, and then the Schr\"odinger operator on $\R^d$ whose absorption index becomes singular for low frequencies. In Section \ref{sec-res-estim-lap} we state and prove the main theorem of this paper about uniform estimates and the limiting absorption principle. Finally we discuss the resolvent estimates for the derivatives of the resolvent in Section \ref{sec-multiple} and the absolutely continuous subspace in Section \ref{sec-Hac}.\\

We close this introduction by some general notation. We set 
\[
\C_+ = \singl{z \in \C \st \Im(z) > 0},
\]
and for $I \subset \R$:
\[
\C_{I,+} = \singl{z \in \C \st \Re(z) \in I, \Im(z) > 0}.
\]
If $\Hc_1$ and $\Hc_2$ are Hilbert spaces, we denote by $\Lc(\Hc_1, \Hc_2)$ the space of bounded operators from $\Hc_1$ to $\Hc_2$.

\bigskip 

\noindent  
{\bf Acknowledgements: } This work was motivated by discussions with David Krej\v{c}i\v{r}\'{\i}k, Petr Siegl and Xue Ping Wang during a French-Czech BARRANDE Project (26473UL). I would like to thank them warmly for their stimulating questions, their helpful remarks and their kind hospitality in Prague/\v Re\v z and Nantes. This research is also partially supported by the French ANR Project NOSEVOL (ANR 2011 BS01019 01).

\section{Dissipative operators and associated conjugate operators} \label{sec-def}

In this section we recall some basic facts about dissipative operators given by form perturbations of self-adjoint operators, and we introduce the corresponding conjugate operators. Let $\Hc$ be a complex Hilbert space.

\begin{definition}
We say that an operator $T$ with domain $\Dom(T)$ on the Hilbert space $\Hc$ is dissipative (respectively accretive) if
\[
\forall \f \in \Dom(T), \quad \Im \innp{T \f}\f _\Hc \leq 0, \quad \big(\text{respectively} \quad \Re \innp{T \f}\f_\Hc \geq 0\big).
\]
Moreover $T$ is said to be maximal dissipative (maximal accretive) if it has no other dissipative (accretive) extension on $\Hc$ than itself. 
\end{definition}

Notice that the conventions for accretive and dissipative operators may be different for other authors.
With our definition, an operator $T$ is (maximal) dissipative if and only if $iT$ is (maximal) accretive. Moreover we recall that a dissipative operator $T$ is maximal dissipative if and only if $(T-z)$ has a bounded inverse on $\Hc$ for some (and hence any) $z \in \C$ with $\Im(z) > 0$. \\

Let $q_0$ be a quadratic form closed, densely defined, symmetric and bounded from below. Let $H_0$ (with domain $\Dom(H_0)$) be the corresponding selfadjoint operator (see \cite[Theorem VI.2.6]{kato}). We denote by $\Kc$ the domain of the form $q_0$ (or the form-domain of the operator $H_0$). We identify $\Hc$ with its dual, and denote by $\Kc^*$ the dual of $\Kc$. Let $\tHo \in \Lc(\Kc,\Kc^*)$ be such that $\innp{\tHo \f} \p _{\Kc^*,\Kc} = q_0(\f,\p)$ for all $\f, \p \in \Kc$.

Let $q_\Th$ be another symmetric form on $\Hc$, non-negative and $q_0$-bounded: there exists $C_\Th \geq 0$ such that for all $\f \in \Kc$ we have
\begin{equation} \label{qI-qO-bounded}
\abs{q_\Th(\f,\f)} \leq C_\Th \left( \abs{q_0(\f,\f)} +   \nr\f ^2 \right).
\end{equation}
We set $q = q_0 - iq_\Th$ and denote by $\tH$ the corresponding operator in $\Lc(\Kc,\Kc^*)$.

\begin{proposition} \label{prop-representation}
There exists a unique maximal dissipative operator $H$ on $\Hc$ such that $\Dom(H) \subset \Dom(q) = \Dom(q_0)$ and
\[
\forall \f \in \Dom(H), \quad \innp{H\f}\f_\Hc =  q(\f, \f).
\]
Moreover, the domain of $H$ is the set of $u \in \Dom(q)$ such that 
\[
 \exists f \in  \Hc, \forall \vf \in \Dom(q), \quad q(u,\vf) = \innp{f}{\vf}_\Hc.
\]
In this case $f$ is unique and we have $Hu = f$.
\end{proposition}

We recall from \cite{art-diss-schrodinger-guide} the following lemma:

\begin{lemma} \label{lem-sect-form}
Let $q_R$ be a non-negative, densely defined, closed form on a Hilbert space $\Hc$. Let $q_I$ be a symmetric form relatively bounded with respect to $q_R$. Then the form $q_R - i q_I$ is sectorial and closed.
\end{lemma}

\begin{proof}[Proof of Proposition \ref{prop-representation}]
There exists $\g \geq 0$ such that $q_0 + \g$ is non-negative. According to Lemma \ref{lem-sect-form}, the form $q_\g := q_0 + \g -i q_\Th$ is sectorial and closed. We denote by $H_\g$ the maximal accretive operator associated to $q_\g$ by the representation theorem (see Theorem VI.2.1 in \cite{kato}). This operator is dissipative. Since it is maximal accretive, $(-1+i)$ belongs to its resolvent set, and hence it is also maximal dissipative. Then it remains to consider $H = H_\g - \g$.
\end{proof}

It is important to note that the form $q_\Th$ is not assumed to be closable, so it is not associated to any operator on $\Hc$. However it defines an operator $\Th \in \Lc(\Kc,\Kc^*)$ and we have
\[
\tH = \tHo - i \Th \quad \text {in } \Lc(\Kc, \Kc^*).
\]
Compared to the setting of \cite{art-mourre}, this equality is not assumed to have a sense in $\Lc(\Dom(H_0),\Hc)$. We will see in Paragraph \ref{sec-guide} an example of operator of this form which cannot be written as $H_{s.a.} - iV$ for $H_{s.a.}$ self-adjoint and $V$ self-adjoint, non-negative and $H_{s.a.}$-bounded with relative bound less than 1. Our purpose in this paper is to recover the results of \cite{art-mourre} in this case.\\

According to the Lax-Milgram Theorem the operators $i(\tH-z)$ and hence $(\tH-z)$ have bounded inverses in $\Lc(\Kc^*,\Kc)$ for all $z \in \C_+$. Moreover we have in $\Lc(\Kc^*,\Kc)$ the resolvent identities 
\begin{equation} \label{eq-res-identity}
\begin{aligned}
(\tH - z)\inv
& = (\tHo-z)\inv + i (\tHo -z)\inv \Th (\tH -z)\inv\\
& = (\tHo -z)\inv + i (\tH -z)\inv \Th (\tHo -z)\inv.
\end{aligned}
\end{equation}
Notice that for $f \in \Hc \subset \Kc^*$ we have 
\[
(H-z)\inv f = (\tH -z)f .
\]

We now introduce the conjugate operator $A$ for $H$. Before the definition we recall from \cite{amreinbg} (see Lemma 1.1.4) the following result:

\begin{lemma} \label{lem-Ec}
Let $A$ be a self-adjoint operator on $\Hc$. Assume that $\Kc$ is left invariant by $e^{-itA}$ for all $t \in \R$. Then the domain of the generator of the unitary group $\restr{e^{-itA}}{\Kc}$ is 
\[
\singl{\f \in \Dom(A) \cap \Kc \st A\f \in \Kc}.
\]
\end{lemma}

Given $t \in \R$, we remark that under the assumption of Lemma \ref{lem-Ec} we can extend by duality the operator $e^{-itA}$ to $\Kc^*$, which is also left invariant.

\begin{definition} \label{def-conj}
Let $A$ be a self-adjoint operator on $\Hc$. We say that $A$ is a conjugate operator (in the sense of forms) to $H$ on the interval $J$ if the following conditions are satisfied:
\begin{enumerate} [(i)]
\item \label{item-Kc-invariant}
The form domain $\Kc$ is left invariant by $e^{-itA}$ for all $t \in \R$. We denote by $\Ec$ the domain of the generator of $\restr{e^{-itA}}{\Kc}$.
\item The commutators $B_0 = [\tilde H_0,iA]$ and $B = [\tilde H,iA]$, \emph{a priori} defined as operators in $\Lc(\Ec,\Ec^*)$, extend to operators in $\Lc(\Kc,\Kc^*)$.
\item The commutator $[B,iA]$, \emph{a priori} defined as an operator in $\Lc(\Ec,\Ec^*)$, extends to an operator in $\Lc(\Kc,\Kc^*)$.
\item There exist $\a > 0$ and $\b \geq 0$ such that 
\begin{equation} \label{hyp-mourre}
\1 J (H_0) (B_0 + \b \Th) \1 J (H_0) \geq \a \1 J (H_0).
\end{equation}
\end{enumerate}
\end{definition}

\begin{remark}
If $H = H_0 -iV$ and $A$ is a conjugate operator for $H$ on $J$ in the sense of Definition 2.3 in \cite{art-mourre} then $H$ can be seen as a perturbation of $H_0$ in the sense of forms and $A$ is a conjugate operator for $H$ on $J$ in the sense of Definition \ref{def-conj}.
\end{remark}

When dealing with a family of operators indexed by a parameter $\l$, it may be important to track the dependance in $\l$ of all the quantities which appear in this definition. In this case we will refer to the following refined version of Definition \ref{def-conj}:

\begin{definition} \label{def-conj-refined}
We say that $A$ is a conjugate operator (in the sense of forms) to $H$ on $J$ and with bounds $(\a,\b,\Upsilon) \in ]0,1] \times \R_+ \times \R_+$ if all the assumptions of Definition \ref{def-conj} are satisfied (in particular $\a$ and $\b$ are the constants which appears in \eqref{hyp-mourre}) and moreover
\[
\nr {B} \leq \sqrt \a \Upsilon , \quad \nr {B + \b \Th} \nr{B_0} \leq \a  \Upsilon \quad \text{and} \quad \nr{[B,A]} + \b \nr{[\Th,A]} \leq \a \Upsilon,
\]
where all the norms are in $\Lc(\Kc,\Kc^*)$.
\end{definition}

These definitions include the assumptions we will need to prove a uniform estimate and the limiting absorption principle for the resolvent of $H$. However it is known that in order to estimate the derivatives of the resolvent we have to control more commutators of $H$ with the conjugate operator $A$:

\begin{definition} \label{def-conj-N}
Let $N \in \N^*$. We set $B_1 = B$. We say that the selfadjoint operator $A$ is a conjugate operator for $H$ on $J$ up to order $N$ if it is a conjugate operator in the sense of Definition \ref{def-conj} and if for all $n \in \Ii 1 N$ the operator $[B_{n},iA]$ defined (inductively) in $\Lc(\Ec,\Ec^*)$ extends to an operator in $\Lc(\Kc,\Kc^*)$, which we denote by $B_{n+1}$. 
\end{definition}

Again, for a family of operators it may be useful to control the size of these multiple commutators:

\begin{definition} \label{def-conj-N-refined}
We say that $A$ is a conjugate operator for $H$ on $J$ with bounds $(\a ,\b ,\Upsilon_N)$ up to order $N$ if it is a conjugate operator for $H$ on $J$ with bounds $(\a ,\b ,\Upsilon)$ in the sense of Definition \ref{def-conj-refined}, if it is a conjugate operator up to order $N$ in the sense of Definition \ref{def-conj-N}, and if 
\begin{equation*} 
\Upsilon + \frac 1 \a \sum_{n=2}^{N+1} {\nr{B_n}_{\Lc(\Kc,\Kc^*)}}  \leq \Upsilon_N .
\end{equation*}
\end{definition}

\section{The dissipative wave guide and other applications} \label{sec-guide}

Before going further, we give some applications to illustrate the definitions of Section \ref{sec-def} and to motivate the upcoming abstract theorems. \\

We first recall that for the free laplacian $-\D$ on $\R^d$ an example of conjugate operator is given by the generator of dilations
\begin{equation} \label{def-A}
A = -\frac i 2 (x \cdot \nabla + \nabla \cdot x) = - i (x\cdot \nabla) - \frac {id}2. 
\end{equation}
Indeed for all $t \in \R$ the dilation $e^{-itA}$ maps $u$ to $e^{-itA} u : x \mapsto e^{-\frac {dt}2} u (e^{-t} x)$. In particular it leaves invariant the form domain $H^1(\R^d)$. Moreover a straightforward computation gives $[-\D,iA] = -2\D$, so $A$ is conjugate to $-\D$ on any interval $J \Subset \R_+^*$ with bound $\a = 2 \inf(J) > 0$. The study of more general Schr\"odinger operators is usually inspired by this model case.\\

The first example that motivated this abstract work is the following: let $\O \subset \R^{d}$ be a wave guide of the form $\O= \R^{p} \times \o$ where $p \in \Ii 1 {d-1}$ and $\o$ is a smooth open bounded subset of $\R^{d-p}$. A general point in $\O$ is denoted by $(x,y)$ with $x \in \R^p$ and $y \in \o$. Let $a \in W^{1,\infty}(\partial \O)$. We consider on $L^2(\O)$ the operator
\begin{equation} \label{def-Ha}
H_a = -\D
\end{equation}
with domain 
\begin{equation} \label{def-Dom-Ha}
\Dom(H_a) = \singl{u \in H^2(\O) , \partial_\n u = i a u \text{ on } \partial \O}.
\end{equation}
We could also consider a (dissipative) perturbation of the free laplacian in the interior of $\O$. This operator appears in the spectral analysis of the wave equation 
\begin{equation} \label{eq-wave}
\begin{cases}
\partial_t ^2 w  - \D w  = 0 & \text{on } \R_+ \times  \O,\\
\partial_\n w + a \partial_t w = 0 & \text{on } \R_+ \times  \partial \O,\\
w(0,\cdot ) = w_0 , \quad \partial_t w (0,\cdot) = w_1 & \text{on }\O,
\end{cases}
\end{equation}
or the Schr\"odinger equation 
\begin{equation} \label{eq-schrodinger}
\begin{cases}
-i \partial_t  u  - \D u  = 0 & \text{on } \R_+ \times  \O,\\
\partial_\n u = i a u  & \text{on } \R_+ \times  \partial \O,\\
u(0,\cdot ) = u_0 & \text{on }\O. 
\end{cases}
\end{equation}
In \cite{art-diss-schrodinger-guide} we have studied \eqref{eq-schrodinger} in the particular case where $\dim \o = 1$ and $a$ is greater than a positive constant at least on one side of the boundary. In this situation it was possible to compute almost explicitely some spectral properties of $H_a$. In particular we proved that $\s(H_a)$ is included in $\singl{z \in \C \st \Im(z) < -\g}$ for some $\g > 0$ with a uniform estimate for the resolvent on the real axis, which gives exponential decay for the solution of \eqref{eq-schrodinger}. When the absorption index $a$ is not that strong, for instance if it is compactly supported on $\partial \O$, the essential spectrum will stay included in the real axis. Then we need more general tools to prove uniform resolvent estimates up to the real axis in this case.

We know from \cite{art-diss-schrodinger-guide} that the operator $H_a$ is maximal dissipative. The corresponding quadratic form is 
\begin{equation} \label{def-qa}
q_a : \f \mapsto \int_\O \abs {\nabla \f}^2  - i \int_{\partial \O} a \abs \f =: q_0 (\f) -i q_\Th (\f).
\end{equation}
It is defined on $\Kc = H^1(\O)$. The self-adjoint part $q_0$ is associated with the operator $H_0$ (defined as $H_a$ with the Neumann boundary condition $a = 0$). However the imaginary part is not associated to any operator on $L^2(\O)$. Since $\Dom(H_a) \neq \Dom(H_a^*) = \Dom(H_{-a})$ there is no hope to write $H_a$ as $H_{s.a.} -iV$ for some self-adjoint operator $H_{s.a.}$ and some non-negative self-adjoint operator $V$ relatively bounded with respect to $H_{s.a.}$ with relative bound less than 1 as is required in \cite{art-mourre, boucletr14}.\\

We define $\tilde H_0$ and $\tilde H_a$ in $\Lc(\Kc,\Kc^*)$ as in Section \ref{sec-def}. Let $L_\o$ denote the Laplacian with Neumann boundary condition on the compact $\o$. $L_\o$ is self-adjoint on $L^2(\o)$ with compact resolvent. We denote by $0 = \l_0 < \l_1 \leq \l_2\leq \dots$ its eigenvalues and by $\seq \f n$ a corresponding sequence of orthonormal eigenfunctions. The spectrum of $H_0$ is given by $\bigcup_{n \in \N} \l_n + \R_+ = \R_+$, and the eigenvalues of $L_\o$ are the thresholds in the spectrum of $H_0$. We denote by $\Tc$ the set of these thresholds. Assume that $u \in \Dom(H_0)$ and $\l \in \R$ are such that $H_0 u = \l u$. Let 
\[
\hat u : (\x, y) \in \R^p \times \o  \mapsto \int_{x \in \R^p} e^{-i\innp x \x} u(x,y) \, dx
\]
be the partial Fourier transform of $u$ with respect to $x$. Then for almost all $\x \in \R^p$ we have
\[
(L_\o + \abs \x^2 - \l) \hat u (\x,\cdot) = 0.
\]
Since $L_\o$ has a discrete set of eigenvalues, $\hat u (\x,\cdot)$ vanishes for $\x$ outside a set of measure 0 in $\R^p$. This proves that $u=0$, and hence $H_0$ has no eigenvalue.

We denote by $\nabla _x$ the gradient with respect to the first $p$ variables on $\O$.
Then we consider the generator $A_x$ of dilations in the first $p$ variables, defined by 
\[ 
(A_x u) (x,y) = -i x\cdot \nabla_x u (x,y) - \frac {ip}2.
\]
Then for $u \in L^2(\O)$, $t \in \R$ and $(x,y) \in \R^p \times \o$ we have
\begin{equation} \label{dilatation-p}
e^{-itA_x} u (x,y) = e^{ -\frac {tp}2} u\left( e^{-t} x,y \right).
\end{equation}

\begin{proposition} \label{prop-guide}
Let $J \subset \R_+^* \setminus \Tc$ be a compact interval. Let $N \in \N^*$. Assume that for $\g \in \N^p$ with $\abs \g \leq N$ we have
\begin{equation} \label{hyp-a-A}
\abs{\partial_{x}^\g a (x,y)} \leq c_\g \pppg {x}^{-\abs{\g}}.
\end{equation}
Then $A$ is conjugate to $\Ha$ on $J$ up to order $N$.
\end{proposition}

\begin{proof}
\stepp According to \eqref{dilatation-p} the form domain $\Kc = H^1(\O)$ is left invariant by $e^{-itA_x}$ for any $t \in \R$. Let $\f,\p \in \Sc(\O)$. If $-\D_x$ denotes the Laplacian in the first $p$ directions we have 
\[
[-\D,iA_x] = [-\D_x,iA_x] = -2\D_x \quad \text{(on $\O$)} \quad \text{and} \quad [a , i A_x] = - (x \cdot \nabla_x) a  \quad \text{(on $\partial \O$)},
\]
so
\begin{align*}
\innp{ [\tilde H_a,iA_x] \f } \p 
= 2 \innp{ \nabla_x \f }{ \nabla_x \p}_{L^2(\O)} + i \int_{\partial \O} (x \cdot \nabla_x a) \f \bar \p .
\end{align*}
With $a = 0$ we simply obtain
\[
\innp{ [\tilde H_0,iA_x] \f } \p = 2 \innp{ \nabla_x \f }{ \nabla_x \p}_{L^2(\O)}.
\]
We similarly compute for any $n \in \Ii 0 N$
\[
\innp{\ad_{iA_x}^n (\tilde H_a) \f}\p = 2^n \innp{ \nabla_x \f }{ \nabla_x \p}_{L^2(\O)} - i \int_{\partial \O} \big((-x \cdot \nabla_x)^n a\big) \f \,\bar \p  .
\]
This implies in particular that the forms $\ad_{iA_x}(\tilde H_0)$ and $\ad_{iA_x}^n(\tilde H_a)$ for $n \in \Ii 1 N$ extend to forms on $H^1(\O)$. It remains to check the last assumption of Definition \ref{def-conj}.

\stepp There exist $m \in \N$ and $\e > 0$ such that $J \subset [\l_m + \e ,\l_{m+1}- \e]$. Let $u \in L^2(\O)$. For almost all $x \in \R^p$ we have $u(x,\cdot) \in L^2(\o)$ so we can find a sequence $(u_n(x))_{n\in\N}$ in $\C^\N$ such that 
\[
u(x,\cdot) = \sum_{n\in\N} u_n(x) \f_n \quad \text{and in particular} \quad \sum_{n\in\N} \abs{u_n(x)}^2 = \nr{u(x,\cdot)}_{L^2(\o)}^2 .
\]
This defines a sequence $\seq u n$ of functions in $L^2(\R^p)$ with
\[
\sum_{n\in\N} \nr{u_n}_{L^2(\R^p)}^2 = \nr u _{L^2(\O)}^2.
\]
With the same proof as for Proposition 4.3 in \cite{art-diss-schrodinger-guide} we can check that for $z \in \C \setminus \R_+$ we have 
\[
(H_0 - z)\inv u = \sum_{n \in \N} (-\D_x + \l_n -z)\inv u_n \otimes \f_n.
\]
Moreover if $u \in \Dom(H_0)$ then $u_n \in H^2(\R^p)$ for all $n \in \N$ and we have 
\[
H_0 u = \sum_{n \in \N} (-\D + \l_n) u_n \otimes \f_n.
\]
Let $n \geq m+1$. We have 
\begin{align*}
\innp{H_0 (u_n \otimes \f_n)} {(u_n \otimes \f_n)}_{L^2(\O)}
& = \innp{(-\D + \l_n) u_n}{u_n}_{L^2(\R^p)} \geq \l_n \nr{u_n}_{L^2(\R^p)}^2\\
& \geq \l_{m+1} \nr{u_n \otimes \f_n}_{L^2(\O)}^2.
\end{align*}
In particular $\1 J (H_0) (u_n \otimes \f_n) = 0$. For a bounded operator $T$ we set $\Im (T) = (T-T^*)/(2i)$. Since $H_0$ and $-\D_x$ have no eigenvalues we can write
\begin{align*}
\1 J (H_0) u
& = \sum_{n = 0}^m \1 J (H_0) (u_n \otimes \f_n)\\
& = \frac 1 {\pi} \lim_{\m\to 0} \sum_{n =0}^m \int_{J} \Im \big(-\D_x + \l_n -(\t+i\m)\big)\inv  u_n \otimes \f_n \, d\t\\
& = \sum_{n =0}^m \1 J (-\D_x + \l_n) (u_n) \otimes \f_n.
\end{align*}
This gives 
\begin{eqnarray*}
\lefteqn{\innp{[\tilde H_0,iA_x] \1 J (H_0) u}{\1 J (H_0) u}_{L^2(\O)} = \innp{-2 \D_x \1 J (H_0) u}{\1 J (H_0) u} _{L^2(\O)}}\\
&\hspace{1cm}& = \sum_{n=0}^m \innp{-2\D_x \1 J(-\D_x + \l_n) (u_n) \otimes \f_n}{\1 J(-\D_x + \l_n) (u_n) \otimes \f_n}_{L^2(\O)}\\
&& \geq 2 \e \sum_{n=0}^m \nr{ \1 J(-\D_x + \l_n) (u_n) \otimes \f_n}^2_{L^2(\O)}\\
&& \geq 2 \e \nr{\1 J (H_0) u}^2_{L^2(\O)}.
\end{eqnarray*}
This proves \eqref{hyp-mourre} with $\a = 2\e$ and concludes the proof of the proposition.
\end{proof}

We could similarly analyse the same problem on the half-space 
\begin{equation} \label{def-half-space}
\O = \singl{(x_1,\dots,x_d) \in \R^d \st x_d > 0}.
\end{equation}
We also mention the Schr\"odinger operator on $\R^d$ with dissipation on the hyperplane $\Sigma = \R^{d-1} \times \singl 0$ given by the transmission condition 
\begin{equation} \label{diss-cond-hypersurface}
\partial_{x_d} u(x',0^+) - \partial_{x_d} u(x',0^-) = - i a(x') u(x',0) \quad \text{on } \Sigma.
\end{equation}
Here we have denoted by $x = (x',x_d)$ a general point in $\R^d$, with $x' \in \R^{d-1}$ and $x_d \in \R$. When $d = 1$ this corresponds to the second derivative with (dissipative) Dirac potential, usually denoted by
\[
u \mapsto -u'' - i a \d(x) u.
\]
More precisely, given $a \in W^{1,\infty}(\Sigma)$ we consider on $L^2(\R^d)$ the operator $H_a = -\D$ with domain
\[
\Dom(H_a) = \singl{u \in H^1(\R^d) \cap H^2(\R^d \setminus \Sigma) \st u \text{ satisfies }\eqref{diss-cond-hypersurface}}.
\]
Given $u \in \Dom(H_a)$ we define $H_a u$ as the function $f \in L^2(\R^d)$ which coincide with the distribution $-\D u$ on $\R^d \setminus \Sigma$. The operator $H_a$ is associated to the quadratic form 
\[
q_a : \f \mapsto \int_{\R^d} \abs{\nabla \f}^2 \, dx - i \int_\Sigma a(x') \abs{u(x',0)}^2 \, dx',
\]
defined on $\Dom(q_a) = H^1(\O)$.

In both cases, we can take the generator of dilations \eqref{def-A} as a conjugate operator on any compact interval $J \subset \R_+^*$ if for all $k \in \N$ the function $(x' \cdot \nabla')^k a$ is bounded on $\partial \O$ or $\Sigma$ (we have denoted by $\nabla '$ the gradient in the first $(d-1)$ variables).\\

In the same spirit as the last example, we can also mention the dissipative quantum graphs with some infinite edges and dissipation at the vertices, given by the condition 
\begin{equation} \label{cond-vertex}
u_1(0) = \dots = u_{n_\n}(0) \quad \text \quad \sum_{j=1}^{n_\n} u'(0) = -ia_\n u(0), 
\end{equation}
where for a vertex $\n$ the integer $n_\n$ is the number of edges attached to $\n$ and $a_\n \geq 0$. For precise definitions we refer to \cite{ong}, which deals with the limiting absorption principle for such a quantum graph with self-adjoint boundary conditions at the vertices (in particular \eqref{cond-vertex} with $a_\n=0$ for all vertices $\n$). For various non-selfadjoint conditions on quantum graphs we also refer to \cite{husseinks14}.\\

We finish this section with the example of the Schr\"odinger operator with dissipation by a potential in low dimensions and for low frequencies. In this case the dissipative Mourre theory in the sense of operators as given in \cite{art-mourre, boucletr14} can be applied, but not uniformly.\\

We consider on $\R^d$, $d \geq 3$ the Schr\"odinger operator 
\[
H_\l = - \D - \frac i {\l^2} a \left( \frac x \l \right),
\]
where $\l > 0$ and $a \in C^\infty(\R^d, \R_+)$ is of very short range: for some $\rho > 0$ there exist constants $c_\a$, $\a \in \N^d$ such that 
\[
\abs{\partial^\g a(x)} \leq c_\g \pppg x^{-2-\rho -\abs \g}.
\]
In order to obtain low frequency resolvent estimates for the Schr\"odinger operator $-\D -ia$ we have to prove uniform resolvent estimates for $H_\l$ close to the spectral parameter 1 uniformly in $\l > 0$ (see \cite{boucletr14} for the wave equation). Since $a$ is bounded the multiplication by $\frac 1 {\l^2} a \left( \frac x \l \right)$ defines a bounded operator on $L^2(\R^d)$ so for any $\l > 0$ we can apply to $H_\l$ the dissipative Mourre theory for perturbations in the sense of operators. However this absorption index becomes singular when $\l$ is close to 0 and it is not clear that this method gives estimates which are uniform in $\l$.\\
According to Proposition 7.2 in \cite{boucletr14} we have for $u \in \Sc$
\[
\nr{a\left( \frac x \l \right) u}_{H^s} \lesssim \l^2\nr{u}_{H^{s+2}}
\]
whenever $s$ and $s + 2$ belong to $\big ]-\frac d 2, \frac d 2\big[$. The same applies if we replace $a$ by $(x\cdot \nabla)^k a$ for some $k \in \N$. This proves that the commutator between the dissipative part of $H_\l$ and the generator of dilations $A$ defines an operator in $\Lc(H^2(\R^d),L^2(\R^d))$ uniformly in $\l > 0$ if $d \geq 5$. But not if $d \in \{ 3 , 4\}$. However, for any $d \geq 3$ it defines a uniformly bounded operator in $\Lc(H^1(\R^d),H\inv (\R^d))$, so it is fruitful to see it as a perturbation of the free laplacian in the sense of forms. This idea will be used (in a more general setting) in \cite{khenissir}.

\section{Uniform resolvent estimate and limiting absorption principle} \label{sec-res-estim-lap}

In this section we prove the uniform resolvent estimates and the limiting absorption principle in the abstract setting:

\begin{theorem} 
\label{th-mourre-form}
Assume that $A$ is a conjugate operator to $H$ on the interval $J$ with bounds $(\a,\b,\Upsilon)$, in the sense of Definition \ref{def-conj-refined}.
\begin{enumerate} [(i)]
\item 
Let $I \subset \mathring J$ be a compact interval and $\d > \frac 12$. Then there exists $C\geq 0$ (which only depends on $C_\Th$, $I$, $J$, $\d$, $\b$ and $\Upsilon$) such that for all $z \in \C_{I,+}$ we have
\begin{equation} \label{estim-res}
\nr{\pppg A^{-\d} (H-z)\inv \pppg A^{-\d}} _{\Lc(\Hc)}  \leq \frac C \a.
\end{equation}
\item
Moreover for all $\l \in \mathring J$ the limit 
\[
\pppg A^{-\d} \big(H-(\l+i0)\big)\inv \pppg A^{-\d} = \lim_{\m \to 0^+} \pppg A^{-\d} \big(H-(\l+i\m)\big)\inv \pppg A^{-\d} 
\]
exists in $\Lc(\Hc)$ and defines a continuous function of $\l$ on $J$ (it is H\"older-continuous of index $\frac {2\d-1}{2\d+1}$ with a constant of size $\a ^{- \frac {4\d}{2\d +1}}$).
\end{enumerate}
\end{theorem}

\begin{remark} \label{rem-adjoint}
Taking the adjoint we obtain the same estimate with $(H-z)\inv$ replaced by ${(H^* - \bar z)\inv}$.
\end{remark}

The rest of this section is devoted to the proof of Theorem \ref{th-mourre-form}. To simplify the notation, the symbol `` $\lesssim $ '' will be used to replace `` $\leq C$ '' where $C$ is a constant which depends on $C_\Th$, $I$, $J$, $\d$, $\b$ and $\Upsilon$. The dependance in $\a \in ]0,1]$, $z \in \C_{I,+}$ and in the parameter $\e$ (which will be introduced in the proof) will always be explicit.\\

Let $\vf \in C_0^\infty(\R,[0,1])$ be supported in $\mathring J$ and equal to 1 on a neighborhood of $I$ (notice that all the estimates below will also depend on the choice of $\vf$). We set $\Phi = \vf(\Ho)$ and $\Phi^\bot = (1-\vf)(\Ho)$. We have
\[
\Phi \in \Lc(\Kc^*, \Kc)  \quad \text{and} \quad \Phi^\bot \in \Lc(\Hc) \cap \Lc(\Kc) \cap \Lc(\Kc^*).
\]
Now let
\[
M_0 =  \Phi (B_0 + \b \Th) \Phi \quad \text{and} \quad M =  \Phi (B + \b \Th) \Phi.
\]
The operators $M_0$ and $M$ are bounded on $\Hc$, and $M_0$ is the self-adjoint part of $M$. After multiplication by $\Phi$ on both sides, assumption \eqref{hyp-mourre} reads
\begin{equation} \label{minor-Me}
M_0 \geq \a \Phi^2.
\end{equation}

The proof of the following lemma is postponed to the end of the section:

\begin{lemma} \label{lem-comm-MeA}
The operator $[M,A]$, \emph{a priori} defined as an operator in $\Lc(\Ec,\Ec^*)$, extends to an operator in $\Lc(\Kc^*,\Kc)$ which we denote by $[M,A]_\Kc$. Moreover we have 
\[
\nr{[M,A]_\Kc}_{\Lc(\Kc^*,\Kc)} \lesssim \a.
\]
\end{lemma}

Let $\e \geq 0$. The operator $H - i\e M$ is maximal dissipative on $\Hc$ with domain $\Dom(H)$, so for $z \in \C_+$ it has a bounded inverse $(H-i\e M-z)\inv$ in $\Lc(\Hc,\Dom(H))$. As above for $\tH$, the operator $(\tH -i\e M -z)\in \Lc(\Kc,\Kc^*)$ has a bounded inverse 
\[
\Gze := (\tH -i\e M -z)\inv \in \Lc(\Kc^*, \Kc).
\]

The Mourre method relies on the so-called quadratic estimates (see Proposition II.5 in \cite{mourre81}). Here we will use the following version:

\begin{proposition} \label{prop-estim-quad-form}
Let $\g_0$ be a quadratic form closed, densely defined, symmetric and bounded from below. Let $P_0$ be the corresponding selfadjoint operator. Let $\Kc_\g$ denote the domain of the form $\g_0$. Let $\g_I$ be a non-negative and $\g_0$-bounded form on $\Hc$. Let $P$ be the maximal dissipative operator associated to the form $\g_0 - i \g_I$, and $\tilde P$ the corresponding operator in $\Lc(\Kc_\g,\Kc_\g^*)$. Let $\g$ a non-negative form on $\Kc_\g$ which satisfies $\g \leq \g_I$. Then for $z \in \C_+$ and $\f \in \Kc_\g^*$ we have 
\[
\g \big( (\tilde P-z)\inv  \f\big) \leq \abs{\innp{(\tilde P-z)\inv \f}{\f}_{\Kc_\g,\Kc_\g^*}}
\quad
\]
and
\[
\g \big( (\tilde P^* - \bar z)\inv  \f\big) \leq \abs{\innp{( \tilde P - \bar z)\inv \f}{\f}_{\Kc_\g,\Kc_\g^*}}.
\]

\end{proposition}

If $\f \in \Hc$ we can replace $\tilde P$ by $P$ in these estimates.

\begin{proof}
For $z \in \C_+$ and $\f \in \Kc_\g^*$ we have
\begin{align*}
\g \big( (\tilde P-z)\inv  \f\big)
& \leq \frac 1 {2i} \innp{ \big( ( \tilde P^* - \bar z) - (\tilde P - z)\big) ( \tilde P-z)\inv  \f}{(\tilde P-z)\inv  \f}_{\Kc_\g^*, \Kc_\g} \\
& \leq \frac 1 {2i} \innp{(\tilde P-z)\inv \f}{ \f}_{\Kc_\g,\Kc_\g^*} - \frac 1 {2i} \innp{ \f}{ (\tilde P-z)\inv \f}_{\Kc_\g^*,\Kc_\g}\\
& \leq \Im \innp{(\tilde P-z)\inv \f}{\f}_{\Kc_\g,\Kc_\g^*}.
\end{align*}
The second estimate is proved similarly.
\end{proof}

\begin{proposition} \label{prop-gze}
Let $\Kc_0$ stand either for $\Kc$ or $\Hc$. Then there exists $\e_0  \in ]0,1]$ (which depends on $C_\Th$, $I$, $J$, $\b$ and $\Upsilon$) such that for $Q\in \Lc(\Kc_0^*)$, $z \in \C_{I,+}$ and $\e \in ]0,\e_0]$ we have
\begin{equation} \label{estim-phibot-gze}
\nr{\Phi^\bot \Gze Q}_{\Lc(\Kc_0^*,\Kc)} \lesssim \nr {Q}_{\Lc(\Kc_0^*)} + \nr{Q^* \Gze Q}_{\Lc(\Kc_0^*,\Kc_0)}^{\frac 12} ,
\end{equation}
\begin{equation} \label{estim-phi-gze}
\nr{\Phi \Gze Q}_{\Lc(\Kc_0^*,\Kc)}  \lesssim \frac {1}   {\sqrt \a \sqrt \e}\nr{Q^* \Gze Q}_{\Lc(\Kc_0^*,\Kc_0)}^{\frac 12} , 
\end{equation}
\begin{equation}\label{estim-gze}
\nr{\Gze Q}_{\Lc(\Kc_0^*,\Kc)} \lesssim \nr {Q} _{\Lc(\Kc_0^*)} + \frac {\nr{Q^* \Gze Q}^{\frac 12}_{\Lc(\Kc_0^*,\Kc_0)}}  {\sqrt \a \sqrt \e},
\end{equation}
and for $\f \in \Kc_0^*$ with $\nr \f_{\Kc_0^*} \leq 1$:
\begin{equation} \label{estim-qTh}
q_\Th \big(\Phi \Gze Q \f\big) + q_\Th \big(\Phi^\bot \Gze Q \f\big) \lesssim\nr Q_{\Lc(\Kc_0^*)}^2 + \nr{Q^* \Gze Q }_{\Lc(\Kc_0^*,\Kc_0)} .
\end{equation}
These estimates also hold if $\Gze$ is replaced by $\Gze^*$ on the left-hand sides.
\end{proposition}

Applied with $Q = \Id_{\Kc^*}$, \eqref{estim-gze} gives an estimate on $\Gze$ alone:

\begin{corollary} \label{cor-nr-gze}
For $z \in \C_{I,+}$ and $\e \in ]0,\e_0]$ we have
\[
\nr{\Gze}_{\Lc(\Kc^*,\Kc)} + \nr{\Gze^*}_{\Lc(\Kc^*,\Kc)} \lesssim \frac 1 {\a \e}.
\]
\end{corollary}

\begin{proof}[Proof of Proposition \ref{prop-gze}]
\stepp Let $z \in \C_{I,+}$. Since $\Phi + \Phi^\bot = 1$, \eqref{estim-gze} is a direct consequence of \eqref{estim-phibot-gze} and \eqref{estim-phi-gze}. Let $\f \in \Kc_0^*$. According to \eqref{minor-Me} and Proposition \ref{prop-estim-quad-form} applied with $Q\f \in \Kc^*$ and the form $\tilde q$ corresponding to $\a \e \Phi^2$ we have
\begin{align*}
\nr{\Phi \Gze Q \f}^2_{\Hc}
& =\frac 1 {\a  \e} \innp {\a \e \Phi^2 \Gze Q \f}{\Gze Q \f} _{\Hc}\\
& \leq  \frac 1 {\a  \e} \abs{\innp {\Gze Q \f}{ Q \f} _{\Kc,\Kc^*}}\\
& \leq \frac 1  {\a  \e} \nr{Q^* \Gze Q}_{\Lc(\Kc_0^*,\Kc_0)} \nr {\f} ^2_{\Kc_0^*}.
\end{align*}
Since $\vf$ is compactly supported in $J$, there exists a constant $c$ which only depends on $J$ such that 
\[
\nr{\Phi \Gze Q \f}^2_{\Kc} \leq  \frac {c}  {\a  \e} \nr{Q^* \Gze Q}_{\Lc(\Kc_0^*,\Kc_0)} \nr {\f} ^2_{\Kc_0^*}.
\]
The same holds with $\Gze$ replaced by $\Gze^*$, and \eqref{estim-phi-gze} is proved.

\stepp 
Since the quadratic form $q_\Th$ is non-negative we can apply the Cauchy-Schwarz inequality: for $\p_1,\p_2 \in \Kc$ we have 
\[
q_\Th(\p_1 + \p_2) \leq q_\Th(\p_1)+ 2 \sqrt{q_\Th(\p_1)}\sqrt{q_\Th(\p_2)} + q_\Th(\p_2) \leq 2 q_\Th(\p_1)  + 2 q_\Th(\p_2).
\]
In particular
\[
q_\Th(\Phi \Gze Q \f) \leq 2  q_\Th(\Gze Q \f) + 2 q_\Th(\Phi ^\bot \Gze Q \f).
\]
According to Proposition \ref{prop-estim-quad-form} we have
\begin{equation} \label{estim-qTh-Gze}
q_\Th  \big(  \Gze Q \f\big)  \leq \abs{\innp {Q^* \Gze Q \f}\f_{\Kc_0,\Kc_0^*}} \leq \nr{Q^* \Gze Q} _{\Lc(\Kc_0^*,\Kc_0)}\nr \f_{\Kc_0^*}^2.
\end{equation}
On the other hand, according to \eqref{qI-qO-bounded}
\begin{equation} \label{estim-qThbot}
\begin{aligned}
q_\Th(\Phi ^\bot \Gze Q \f)
\leq C_\Th \nr{\Phi^\bot \Gze Q \f}^2_\Kc.
\end{aligned}
\end{equation}
We obtain
\begin{equation} \label{estim-qThPhi}
\begin{aligned}
q_\Th(\Phi \Gze Q \f)
\leq 2 \nr{Q^* \Gze Q}_{\Lc(\Kc_0^*,\Kc_0)} \nr\f^2_{\Kc_0} + 2 C_\Th \nr{\Phi^\bot \Gze Q \f}^2_\Kc.
\end{aligned}
\end{equation}
Thus we have to prove \eqref{estim-phibot-gze} to prove \eqref{estim-qTh}. The proof of \eqref{estim-qTh} relies itself on \eqref{estim-qThPhi}.

\stepp
According to the resolvent identity (as in \eqref{eq-res-identity}) we have in $\Lc(\Kc_0^*,\Kc)$
\begin{align*}
\Phi^\bot \Gze Q = \Phi^\bot (\tHo - z)\inv   Q+ i \Phi^\bot (\tHo - z)\inv \big(\Th + \e \Phi B \Phi + \e \b \Phi \Th \Phi \big) \Gze Q.
\end{align*}
By functional calculus the operator $\Phi^\bot (\tHo-z)\inv$ belongs to $\Lc(\Kc^*,\Kc)$ uniformly in $z \in \C_{I,+}$.
Let $\f \in \Kc_0^*$ and $\p \in \Kc^*$. According to the Cauchy-Schwarz inequality we have
\begin{align*}
\innp{\Phi^\bot (\tHo - z)\inv \Th  \Gze Q \f}{\p}_{\Kc,\Kc^*}
& = q_\Th  \big(  \Gze Q \f, \Phi^\bot (\tHo - \bar z)\inv \p\big) \\
& \leq q_\Th  \big(  \Gze Q \f\big) ^{\frac 12} q_\Th  \big(\Phi^\bot (\tHo - \bar z)\inv \p\big)^{\frac 12}.
\end{align*}
According to \eqref{qI-qO-bounded} we have
\[
q_\Th \left(\Phi^\bot (\tHo - \bar z)\inv  \p\right) \lesssim \nr \p_{\Kc^*}^2.
\]
With \eqref{estim-qTh-Gze} this proves that 
\[
\nr{\Phi^\bot (\tHo - z)\inv \Th  \Gze Q}_{\Lc(\Kc_0^*,\Kc)} \lesssim \nr{Q^* \Gze Q}^{\frac 12}_{\Lc(\Kc_0^*,\Kc_0)}.
\]
Then we have
\begin{align*}
\e \nr{\Phi^\bot (\tHo - z)\inv  \Phi B \Phi \Gze Q}_{\Lc(\Kc_0^*,\Kc)}
& \lesssim \sqrt \a   \e \nr{\Phi \Gze Q}_{\Lc(\Kc_0^*,\Kc)}\\
& \lesssim \sqrt \e\nr{Q^* \Gze Q}^{\frac 12}_{\Lc(\Kc_0^*,\Kc_0)}.
\end{align*}
On the other hand, according to the Cauchy-Schwarz inequality and \eqref{estim-qThPhi} we have
\begin{eqnarray*}
\lefteqn{ \e \b \nr{\Phi^\bot (\tHo - z)\inv  \Phi \Th \Phi \Gze Q \f}_{\Kc}}\\
&& \leq \e \b q_\Th\big(\Phi^\bot (\tHo - \bar z)\inv  \Phi \f\big) ^{\frac 12} q_\Th\big(\Phi \Gze Q \f\big)^{\frac 12} \\
&& \lesssim  \e \left( \nr{Q^* \Gze Q}^{\frac 12}_{\Lc(\Kc_0^*,\Kc_0)} + \nr{\Phi^\bot \Gze Q }_{\Lc(\Kc_0^*,\Kc)} \right) \nr{\f}_{\Kc_0^*}^2.
\end{eqnarray*}
Finally we obtain
\begin{eqnarray*}
\lefteqn{\nr{\Phi^\bot \Gze Q}_{\Lc(\Kc_0^*,\Kc)}}\\
&&\lesssim \left( \nr Q _{\Lc(\Kc_0^*)} + \nr{Q^* \Gze Q}^{\frac 12}_{\Lc(\Kc_0^*,\Kc_0)} + \e  \nr{\Phi^\bot \Gze Q}_{\Lc(\Kc_0^*,\Kc)}\right).
\end{eqnarray*}
This gives \eqref{estim-phibot-gze} when $\e > 0$ is small enough. Then \eqref{estim-qThPhi} and \eqref{estim-qThbot} give \eqref{estim-qTh}.
\end{proof}

\begin{lemma} \label{lem-GMAG}
On $\Lc(\Dom(A),\Dom(A)^*)$ we have 
\begin{align*} 
 \Gze B \Gze  = i  A \Gze - i \Gze A  - \e  \Gze [M,A]_\Kc \Gze .
\end{align*}
\end{lemma}

\begin{proof}
Let $\f,\p \in \Dom(A)$. Since $\Ec$ is dense in $\Kc$ we can consider sequences $\seq \f n$ and $\seq \p n$ in $\Ec$ such that $\f_n \to \Gze \f$ and $\p_n \to \Gze^* \p$ in $\Kc$. Since $B\in \Lc(\Kc,\Kc^*)$ we have 
\[
\innp{B \f_n}{\p_m} \limt {n,m} {\infty} \innp{ \Gze B \Gze  \f} \p.
\]
On the other hand, since $\f_n, \p_m \in \Dom(A)$ and $A\f_n,A\p_m \in \Kc$ we can write
\[
\innp{B \f_n}{\p_m} = \innp{ [\tH,iA] \f_n}{\p_m} = \innp{ [\tH -i\e M -z,iA] \f_n}{\p_m} - \e \innp{ [M,A] \f_n}{\p_m}.
\]
According to Lemma \ref{lem-comm-MeA} we have
\[
 \innp{ [M,A] \f_n}{\p_m} \limt {n,m} \infty \innp{ \Gze [M,A]_\Kc\Gze  \f}{\p}.
\]
And finally
\begin{eqnarray*}
\lefteqn{\lim_{n,m \to \infty} \innp{ [\tH-i \e M -z,iA] \f_n}{\p_m}}\\
&& = i \lim_{n \to \infty} \lim _{m\to \infty} \innp{A\f_n}{(\tH -i \e M -z)^* \p_m}_{\Kc,\Kc^*} - i \lim _{m\to \infty}\lim_{n \to \infty} \innp{ (\tH -i\e M -z)\f_n}{A \p_m}_{\Kc^*,\Kc}\\
&& =  i\lim_{n \to \infty}  \innp{A\f_n}{ \p}_{\Kc,\Kc^*}-  i\lim _{m\to \infty} \innp{\f}{A \p_m}_{\Kc^*,\Kc}\\
&& =  i\lim_{n \to \infty}  \innp{A\f_n}{ \p}_{\Hc} -  i\lim _{m\to \infty} \innp{\f}{A \p_m}_{\Hc}\\
&& = i \lim_{n \to \infty}  \innp{\f_n}{A  \p}_{\Hc} - i\lim _{m\to \infty} \innp{A\f}{\p_m}_{\Hc}\\
&& = i \innp{\Gze  \f}{A  \p}_{\Hc} - i \innp{A\f}{\Gze^* \p}_{\Hc}.
\end{eqnarray*}
The lemma is proved.
\end{proof}

The strategy for the proof of Theorem \ref{th-mourre-form} is standard and relies on the following abstract result about ordinary differential equations (see Lemma 3.3 of \cite{jensenmp84}):

\begin{lemma} \label{lem-JMP}
Let $X$ be a Banach space, $\e_0 \in]0,1]$ and $f \in C^1 (]0,\e_0] , X)$. Suppose there exist $\g_1 \in [0,1]$, $\g_2 \in [0,1[$, $\g_3 \in \R$, and $c_1,c_2 > 0$ such that
\begin{equation*} 
\forall \e \in ] 0, \e_0[,\quad \nr{f'(\e)} \leq c_1 \e ^{-\g_2} (1 + \nr{f(\e)}^{\g_1} ) \quad \text{and} \quad \nr{f(\e)} \leq c_2\e^{-\g_3}.
\end{equation*}
Then $f$ has a limit at 0 and there exists $c \geq 0$ which only depends on $\e_0$, $\g_1$, $\g_2$, $\g_3$, $c_1$ and $c_2$ such that
\begin{equation*} 
\forall \e \in ]0, \e_0[,\quad \nr{f(\e)} \leq c.
\end{equation*}
\end{lemma}

Now we can prove Theorem \ref{th-mourre-form}:

\begin{proof}[Proof of Theorem \ref{th-mourre-form}]
\stepp 
For $\e \in ]0,1]$ we set $Q(\e) = \pppg {A}^{-\d} \pppg {\e A}^{\d -1}$. According to the functional calculus we have
\begin{equation} \label{estim-Qe}
\nr{Q(\e)}_{\Lc(\Hc)} \leq 1 \quad \text{and} \quad \nr{A Q(\e)}_{\Lc(\Hc)} + \nr{Q(\e) A}_{\Lc(\Hc)} \lesssim \e^{\d-1}.
\end{equation}
Denoting by a prime the derivative with respect to $\e$ we also have
\begin{equation} \label{estim-der-Qe}
\nr{Q'(\e)}_{\Lc(\Hc)} \lesssim \e^{\d-1}.
\end{equation}

\stepp
For $z \in \C_{I,+}$ we set
\[
 F_z(\e) = Q(\e) \Gze Q(\e).
\]
According to \eqref{estim-Qe} and Proposition \ref{prop-gze} applied with $Q = Q(\e)$ we have for $\e \in ]0,\e_0]$ ($\e_0$ being given by Proposition \ref{prop-gze})
\begin{equation} \label{estim-gzeQ-F}
\nr{F_z(\e)} \leq \nr{\Gze Q(\e)} \lesssim 1 + \frac {\nr{F_z(\e)}^{\frac 12}} {\sqrt \a \sqrt \e},
\end{equation}
and hence
\begin{equation} \label{estim-Fe}
\nr{F_z(\e)} \lesssim \frac 1 {\a \e}.
\end{equation}

\stepp
We now estimate the derivative of $F$:
\[
F'_z(\e) = Q'(\e) \Gze Q(\e) + Q(\e) \Gze Q'(\e) +  i Q(\e) G(\e) \Phi (B + \b \Th) \Phi G (\e) Q(\e) 
\]
Proposition \ref{prop-gze} and \eqref{estim-der-Qe} yield
\begin{equation} \label{estim-QGQ'}
\nr{Q'(\e) \Gze Q(\e) + Q(\e) \Gze Q'(\e)} \lesssim \e^{\d-1}\left(1 + \frac { \nr{F_z(\e)}^{\frac 12}}{\sqrt \a \sqrt \e} \right)
\end{equation}
and 
\begin{equation} \label{estim-QGTh}
\nr{Q(\e) G(\e) \Phi \Th \Phi G (\e) Q(\e)} \lesssim 1 + \nr{F_z(\e)}_{\Lc(\Hc)}. 
\end{equation}
For the remaining term we write in $\Lc(\Kc,\Kc^*)$
\[
\Phi B \Phi = B - \Phi B \Phi^\bot - \Phi^\bot B \Phi -  \Phi^\bot B \Phi^\bot.
\]
According to Proposition \ref{prop-gze} we have
\[
\nr{Q(\e) G(\e) \big(\Phi B \Phi^\bot + \Phi^\bot B \Phi +  \Phi^\bot B \Phi^\bot  \big)  G (\e) Q(\e) } \lesssim 1+ \frac {\nr{F_z(\e)}}{\sqrt \e}.
\]

\stepp
According to Lemma \ref{lem-GMAG} we have on $\Lc(\Hc)$:
\begin{equation*} 
\begin{aligned} 
Q(\e) \Gze B \Gze Q(\e) 
& = i Q(\e) A \Gze Q(\e) - i Q(\e) \Gze A Q(\e) \\
& \quad - \e Q(\e) \Gze [M,A]_\Kc \Gze Q(\e).
\end{aligned}
\end{equation*}
With \eqref{estim-Qe}, Proposition \ref{prop-gze} and Lemma \ref{lem-comm-MeA} we get
\[
\nr{ Q(\e) \Gze B \Gze Q(\e) } \lesssim 1 + \a^{-\frac 12} \e ^{\d - \frac 32} \nr{F_z(\e)}^{\frac 12} + \nr{F_z(\e)}. 
\]
Together with \eqref{estim-QGQ'} and \eqref{estim-QGTh} this gives
\begin{equation} \label{estim-der-fze}
\nr{\a F'_z(\e)} \lesssim \e^{\d-1} + \e^{-\frac 12} \nr{\a F_z(\e)} +  \e^{\d-\frac 32}\nr{\a F_z(\e)}^{\frac 12},
\end{equation}
and hence, according to Lemma \ref{lem-JMP}, we finally obtain
\begin{equation} \label{estim-fze-fin}
\nr{F_z(\e)} \lesssim \frac 1 \a,
\end{equation}
which gives the uniform resolvent estimates \eqref{estim-res} when $\e$ goes to 0.

\stepp 
Now we prove the limiting absorption principle on $I$. Without loss of generality we can assume that $\d \in \big] \frac 12, 1 \big]$. We prove that there exists $C \geq 0$ such that for all $z,z' \in \C_{I,+}$ we have
\begin{equation} \label{estim-res-diff}
\nr{\pppg A ^{-\d} \big( (H-z)\inv - (H-z')\inv \big) \pppg A^{-\d}}_{\Lc(\Hc)} \lesssim \a ^{-\frac {4\d}{2\d+1}} \abs{z - z'}^{\frac {2\d-1}{2\d+1}}.
\end{equation}
For any $c_0 > 0$, \eqref{estim-res-diff} is a direct consequence of the uniform estimate \eqref{estim-res} as long as $\abs{z-z'} \geq c_0 \a$, so it is enough to prove \eqref{estim-res-diff} when $\abs {z-z'} \leq c_0 \a$ for some  well chosen $c_0 > 0$. According to \eqref{estim-der-fze} and \eqref{estim-fze-fin} we have
\[
\nr{F_z'(\e)} \lesssim \a \inv \e ^{\d - \frac 32}, 
\]
and hence 
\[
\nr{F_z(\e) - F_z(0)} \lesssim \a \inv \e ^{\d - \frac 12}.
\]
Of course we have the same estimate for $z'$. Moreover, according to \eqref{estim-gzeQ-F} we have for all $\e \in ]0,\e_0]$
\[
\nr{\frac {\partial} {\partial z} F_z(\e)} = \nr{Q(\e) \Gze^2 Q(\e)} \leq \nr{\Gze Q(\e)} \lesssim \frac 1 {\a^2 \e},
\]
and hence
\[
\nr{F_z(\e) - F_{z'}(\e)} \lesssim \frac { \abs{z-z'} }{\a^2 \e}.
\]
Given $z$ and $z'$ we take
\[
\e = \left( \frac {\abs {z-z'}}\a \right)^{\frac 2 {2\d + 1}}.
\]
If $c_0$ was chosen small enough then $\e \in ]0,\e_0]$, and we obtain
\[
\nr{F_z(0) - F_{z'}(0)} \lesssim \a ^{-\frac {4\d}{2\d+1}} \abs{z-z'}^{\frac {2\d-1}{2\d+1}},
\]
which is exactly \eqref{estim-res-diff}. Now for all $ \l \in I$ the function
\[
\m \mapsto \pppg A^{-\d} \big( H - (\l+i\m)\big)\inv \pppg A^{-\d}
\]
has a limit when $\m$ goes to $0^+$. Taking the limit $\Im z, \Im z' \to 0^+$ in \eqref{estim-res-diff} proves that this limit is a H\"older-continuous function of index $\frac {2\d-1}{2\d+1}$.
\end{proof}

To conclude we have to give a proof of Lemma \ref{lem-comm-MeA}:

\newcommand{\Halph}{\tilde H_\th}\newcommand{\Hbeta}{\tilde H_\t}

\begin{proof} [Proof of Lemma \ref{lem-comm-MeA}] The proof is inspired by the proof of Lemma 1.2.1 in \cite{amreinbg}.

\stepp 
For $\th \in \R$ we set
\[
\Halph = e^{i\th A} \tHo e^{-i\th A}  \in \Lc(\Kc,\Kc^*).
\]
We first prove that the map $\th \mapsto \Halph$ is strongly $C^1$ and that for all $\th,\t \in \R$ and $\f \in \Kc$ we have in $\Kc^*$
\begin{equation} \label{notes128}
\big(\Hbeta - \Halph\big)\f = - \int_\th^\t e^{is A} B_0 e^{-is A}\f \, ds.
\end{equation}
This gives in particular 
\begin{equation} \label{estim-tildeH}
\nr{\Hbeta - \Halph}_{\Lc(\Kc,\Kc^*)} \lesssim \abs{\t - \th} \nr{B_0}_{\Lc(\Kc,\Kc^*)}.
\end{equation}
Let $\th \in \R$ and $\f \in \Ec$. For $\e \in \R^*$ we have
\begin{equation} \label{notes1210}
\frac {\tilde H_{\th+\e} - \Halph}{\e} \f = e^{i(\th+\e)A} \tHo \frac {e^{-i\e A} - 1} \e  e^{-i\th A} \f + e^{i \th A}  \frac {e^{i\e A} - 1}\e \tHo   e^{-i\th A} \f.
\end{equation}
Since $e^{-i\th A} \f \in \Ec$ we have
\[
\frac {e^{-i\e A} - 1} \e e^{-i\th A} \f \xrightarrow[\e \to 0]{\Kc} -iA e^{-i\th A} \f,
\]
and hence the first term in the right-hand side of \eqref{notes1210} goes to $-i e^{i\th A} \tHo A e^{-i\th A}\f$ in $\Kc^*$ when $\e$ goes to 0.
Now let $g = \tHo   e^{- i \th A} \f \in \Kc^*$. Since $\Dom(A)$ is dense in $\Kc^*$, we can consider a sequence $\seq g n \in \Dom(A)^\N$ such that $g_n \to g$ in $\Kc^*$. For all $n\in\N$ we have in $\Hc$:
\[
\frac {e^{i\e A} - 1}{\e} g_n -iA g_n = \frac i \e \int_0^\e(e^{i\t A}-1) A g_n \, d\t.
\]
In $\Ec^*$ we can let $n$ go to infinity (we use the Lebesgue dominated convergence theorem for the right-hand side). We obtain that the equality holds in $\Ec^*$ when $g_n$ is replaced by $g$,
and hence the second term in the right-hand side of \eqref{notes1210} goes to $i e^{i \th A} A \tHo e^{-i\th A} \f$ in $\Ec^*$. This proves that the map $\th \mapsto \Halph \f$ is differentiable with derivative $- e^{i\th A} [\tHo,iA] e^{-i\th A}\f \in \Ec^*$, and hence \eqref{notes128} holds in $\Lc(\Ec,\Ec^*)$. Since $B_0 = [\tHo,iA]$ extends to an operator in $\Lc(\Kc,\Kc^*)$, this is the case for both terms in \eqref{notes128} and we have the equality in $\Lc(\Kc,\Kc^*)$.

\stepp
On $\Lc(\Kc,\Kc^*)$ we have $[\tHo , e^{i\th A}] = (\tHo - \Halph)e^{i\th A}$ and hence for $t \in \R$ and $\th \in \R^*$ we have in the strong sense in $\Lc(\Kc,\Kc^*)$:
\begin{align*}
e^{it \Ho} \frac {e^{i\th A} - 1}{i\th} -  \frac {e^{i\th A} - 1}{i\th} e^{it \Ho}
& = \frac 1{i\th} \int_0^t e^{is\Ho} [i\tHo , e^{i\th A}] e^{i(t-s)\Ho}\,ds\\
& = \int_0^t e^{is\Ho} \frac {\tHo - \Halph}\th e^{i\th A} e^{i(t-s) \Ho}\, ds.
\end{align*}
The operator $e^{i\th A}$ goes strongly to 1 in $\Lc(\Kc)$ and $\frac {e^{i\th A} - 1} {i\th}$ converges strongly to $A$ in $\Lc(\Ec,\Kc)$ and $\Lc(\Kc^*, \Ec^*)$.
Moreover
\[
e^{is\Ho} \frac {\tHo - \Halph}\th e^{i\th A} e^{i(t-s) \Ho}
\]
is uniformly bounded in $\Lc(\Kc,\Kc^*)$ according to \eqref{estim-tildeH}. Since $\frac {\tHo - \Halph}\th \to B_0$ strongly in $\Lc(\Kc,\Kc^*)$ when $\th$ goes to 0 (see \eqref{notes128}) we can apply Lebesgue dominated convergence to obtain
\[
 [e^{it\Ho},A] = \int_0^t e^{is\Ho} B_0 e^{i(t-s)\Ho} \, ds,
\]
in the strong sense in $\Lc(\Ec,\Ec^*)$. But the right-hand side defines an operator in $\Lc(\Kc,\Kc^*)$, so the operator on the left has an extension in $\Lc(\Kc,\Kc^*)$ and
\begin{equation} \label{estim-eHo-A}
\nr{ [e^{it\Ho},A]}_{\Lc(\Kc,\Kc^*)} \lesssim \abs t \nr{B_0}_{\Lc(\Kc,\Kc^*)}.
\end{equation}

\stepp Let $\p : x \mapsto \vf(x) (x-i)^2$ and $\Psi = \p(H_0)$. We have $ \Phi= (\Ho-i)\inv \Psi(\Ho-i)\inv$. On $\Lc(\Ec,\Ec^*)$ we have 
\[
[\Psi , A] = \frac 1 {\sqrt {2\pi}} \int_\R [e^{it\Ho} , A] \hat \p (t)\,dt.
\]
The right-hand side extends to an operator in $\Lc(\Kc,\Kc^*)$. Then this is also the case for the left-hand side, and moreover
\[
\nr{[\Psi , A]}_{\Lc(\Kc,\Kc^*)} \lesssim  \nr{B_0}_{\Lc(\Kc,\Kc^*)}\int_\R \abs {t \hat \p(t)}\,dt.
\]
Then
\begin{align*}
[\Phi,A]
& = [(\Ho-i)\inv,A] \Psi (\Ho-i)\inv + (\Ho-i)\inv[\Psi ,A]  (\Ho-i)\inv\\
&\quad  + (\Ho-i)\inv \Psi [(\Ho-i)\inv,A] .
\end{align*}
Since 
\[
[(\Ho-i)\inv,A] = i(\Ho-i)\inv B_0 (\Ho-i)\inv \in \Lc(\Kc^*, \Kc),
\]
this proves that $[\Phi, A] \in \Lc(\Kc^*,\Kc)$ and
\[
\nr{[\Phi, A]} _{\Lc(\Kc^*,\Kc)} \lesssim \nr{B_0}_{\Lc(,\Kc,\Kc^*)}
\]

\stepp Now it only remains to write
\[
[M_0 , A] =  [\Phi,A] (B + \b \Th) \Phi + \Phi [B+\b\Th,A]  \Phi + \Phi (B+\b \Th)   [\Phi,A]
\]
to conclude the proof.
\end{proof}

\section{Multiple commutator estimates} \label{sec-multiple}

\newcommand{\GTh}{G_z^\Th(\e)}
\newcommand{\Gzpe}{G_z^\bot(\e)}
\newcommand{\Gzue}{G_z^1(\e)}
\newcommand{\Gzne}{G_z^n(\e)}
\newcommand{\GzNe}{G_z^N(\e)}
\newcommand{\Gzn}{G_z^n}

In this section we generalize the multiple resolvent estimates known for a self-adjoint operator (see \cite{jensenmp84, jensen85}) or for the perturbation by a dissipative operator (see \cite{art-mourre, boucletr14}).

Let $N\geq 2$ be fixed for all this section. We will use the notation of Definition \ref{def-conj-N-refined}. Thus the symbol `` $\lesssim $ '' will stand for `` $\leq C$ '' where $C$ is a constant which depends on $C_\Th$, $I$, $J$, $\d$, $\b$ and $\Upsilon_N$.\\

For $n \in \Ii 1 N$ and $\e \in ]0,1]$ we set
\begin{equation} \label{def-Cn}
C_n(\e) = \sum_{j=1}^n \frac {(-i\e)^j}{j!}  B_j  \in \Lc(\Kc,\Kc^*).
\end{equation}
In order to prove multiple resolvent estimates, we first need some estimates for the inverse of $\big( \tH+C_n(\e)-z\big)$. It is not clear that this operator has an inverse, since for $n \geq 3$ there is an anti-dissipative term in $C_n(\e)$, but it will be the case for $\e$ small enough. The following result generalizes Lemma 3.1 in \cite{jensenmp84} (see also Lemma 3.1 in \cite{art-mourre}) to our setting:

\begin{proposition} \label{prop-gnze}
Suppose $A$ is a conjugate operator for $H$ up to order $N$ on $J$ with bounds $(\a,\b,\Upsilon_N)$.
\begin{enumerate} [(i)]
\item There exists $\e_N > 0$ such that for $n\in \Ii 1 N$, $z \in \C_{I,+}$ and $\e \in ]0,\e_N]$ the operator $\big(\tH + C_n(\e) - z\big)$ has a bounded inverse in $\Lc(\Kc^*,\Kc)$, which we denote by $\Gzne$.
\item For $n\in\N$, $z \in \C_{I,+}$ and $\e \in ]0,\e_N]$ we have
\[
\nr{\Gzne}_{\Lc(\Kc^*,\Kc)} \lesssim \frac {1}{\a \e}
\]
and 
\[
\nr{\Gzne \pppg A \inv}_{\Lc(\Hc,\Kc)} \lesssim \frac {1}{\a \sqrt \e}.
\]

\item The function $\e \in ]0,\e_N[ \mapsto \Gzne$ is differentiable in $\Lc(\Kc^*,\Kc)$. Moreover in $\Lc(\Dom(A),\Dom(A)^*)$ we have the equality
\[
\frac d {d\e} \Gzne = \left[ \Gzne , A \right] -i \frac {(-i\e)^n}{n!} \Gzne B_{n+1} \Gzne.
\]

\end{enumerate}

\end{proposition}

For the proof of Proposition \ref{prop-gnze} we need the following lemma, inspired by the standard technique for factored perturbations (see \cite{kato66}):

\begin{lemma} \label{lem-res-perturb}
Let $T \in \Lc(\Kc,\Kc^*)$ and assume that $T$ has an inverse $T\inv \in \Lc(\Kc^*,\Kc)$. Let $P_1 \in \Lc(\Hc,\Kc^*)$ and $P_2 \in \Lc(\Kc,\Hc)$ be such that $\nr{P_2 T\inv P_1}_{\Lc(\Hc)} < 1$. Then $T + P_1P_2 \in \Lc(\Kc,\Kc^*)$ has a bounded inverse given by $T\inv - T\inv P_1 \G\inv P_2 T\inv \in \Lc(\Kc^* ,\Kc)$, where $\G = 1 + P_2 T\inv P_1 \in \Lc(\Hc)$.
\end{lemma}

\begin{proof} [Proof of Lemma \ref{lem-res-perturb}]
The assumptions ensure that $\G$ is bounded on $\Hc$ with bounded inverse, so the operator $R = T\inv - T\inv P_1 \G\inv P_2 T\inv$ is well-defined in $\Lc(\Kc^* ,\Kc)$. We only have to check that $R$ is indeed an inverse for $T + P_1P_2$. On $\Kc^*$ we have
\begin{align*}
(T+P_1P_2) R 
& = 1 + P_1P_2 T\inv - P_1 \G\inv P_2 T\inv - P_1P_2 T\inv P_1 \G\inv P_2 T\inv \\
& = 1 + P_1 \big(1 - \G\inv - P_2T\inv P_1 \G\inv\big) P_2 T\inv\\
& = 1 + P_1 \big(1 - (1 + P_2T\inv P_1)\G\inv \big) P_2 T\inv\\
& = 1.
\end{align*}
Similarly we have on $\Kc$:
\begin{align*}
R (T + P_1P_2)
& = 1 + T\inv P_1P_2 - T\inv P_1 \G\inv P_2 - T\inv P_1 \G\inv P_2 T\inv P_1P_2 \\
& = 1 + T\inv P_1 \big( 1 - \G \inv - \G\inv P_2T\inv P_1  \big) P_2 \\
& = 1. \qedhere
\end{align*}
\end{proof}

\begin{proof} [Proof of Proposition \ref{prop-gnze}]
We use the notation introduced in Section \ref{sec-res-estim-lap}.

\stepp
Let $\e_0 > 0$ be given by Proposition \ref{prop-gze}. The operator $\Phi \Th \Phi$ is bounded and self-adjoint on $\Hc$. It is also non-negative, so its square root $\sqrt {\Phi \Th \Phi}$ is well-defined as a bounded operator on $\Hc$.
As in Proposition \ref{prop-gze}, we write $\Kc_0$ either for $\Hc$ or $\Kc$. Then for $Q \in \Lc(\Kc_0^*)$, $z \in \C_{I,+}$, $\e \in ]0,\e_0]$ and $\f \in \Kc_0^*$ we have according to Proposition \ref{prop-gze}:
\begin{align*}
\innp { \Th \Phi \Gze Q \f}  {\Phi \Gze Q \f} \lesssim \left(\nr{Q}_{\Lc(\Kc_0^*)}^2 + \nr{Q^* \Gze Q}_{\Lc(\Kc_0^* ,\Kc_0)}\right) \nr{\f}_{\Kc_0^*}^2.
\end{align*}
This proves that 
\begin{equation} \label{estim-sqrtTh-Q}
 \nr{\sqrt {\Phi \Th \Phi} \Gze Q} _{\Lc(\Kc_0^*,\Kc_0)} \lesssim \nr{Q}_{\Lc(\Kc_0^*)} + \nr{Q^* \Gze Q}^{\frac 12}_{\Lc(\Kc_0^*,\Kc_0)}.
\end{equation}
Applied with $Q = \sqrt {\Phi \Th \Phi} \in \Lc(\Hc)$, this gives
\begin{equation} \label{norm-sqrt-sqrt}
\sup_{\substack{z \in \C_{I,+} \\ \e \in ]0,\e_0]}} \nr{ \sqrt {\Phi \Th \Phi} \Gze \sqrt {\Phi \Th \Phi}}_{\Lc(\Hc)} < + \infty.
\end{equation}

\stepp For $z \in \C_{I,+}$ and $\e \in ]0,\e_\Th]$ (where $\e_\Th \in ]0,\e_0]$ is chosen small enough) we can apply Lemma \ref{lem-res-perturb} with $T = (\tH-i\e M -z) \in \Lc(\Kc,\Kc^*)$, $P_1 = i\e \b \sqrt {\Phi \Th \Phi} \in \Lc(\Hc)$ and $P_2 = \sqrt {\Phi \Th \Phi}\in \Lc(\Hc)$. We obtain that the operator $(\tH-i\e  \Phi B \Phi -z)$ has a bounded inverse $\GTh \in \Lc(\Kc^*,\Kc)$, given by
\begin{equation} \label{def-GTh}
\GTh = \Gze - i \e \b \Gze \sqrt {\Phi \Th \Phi} \G^\Th_{z}(\e) \inv  \sqrt {\Phi \Th \Phi} \Gze ,
\end{equation}
where 
\[
\G^\Th_{z}(\e) =  1 + i \e \b \sqrt {\Phi \Th \Phi} \Gze\sqrt {\Phi \Th \Phi} \in \Lc(\Hc).
\]
In particular $\G^\Th_{z}(\e) \inv$ is bounded in $\Lc(\Hc)$ uniformly with respect to $z \in \C_{I,+}$ and $\e \in ]0,\e_\Th]$. Corollary \ref{cor-nr-gze} and estimate \eqref{estim-sqrtTh-Q} applied with $Q = \Id_{\Kc^*}$ give
\[
 \nr{\sqrt {\Phi \Th \Phi} \Gze} _{\Lc(\Kc^*,\Hc)}  \lesssim \frac 1 {\sqrt \a \sqrt \e}.
\]
With the similar estimate for $\Gze \sqrt {\Phi \Th \Phi}$ and \eqref{def-GTh} we obtain
\begin{equation} \label{estim-GTh}
\nr{\GTh} _{\Lc(\Kc^*,\Kc)} \lesssim \frac 1 {\a \e}.
\end{equation}
With Proposition \ref{prop-gze} we can check similarly that 
\begin{equation} \label{estim-GTh-bot}
\nr{\GTh  \Phi^\bot}_{\Lc(\Kc^*,\Kc)} \lesssim \frac 1 {\sqrt{\a} \sqrt \e}
\end{equation}
and
\begin{equation} \label{estim-GTh-Ainv}
\nr{\GTh  \pppg {A}\inv }_{\Lc(\Hc,\Kc)}  \lesssim \frac 1 { {\a}\sqrt \e}.
\end{equation}

\stepp Now we want to apply Lemma \ref{lem-res-perturb} with $T = (\tH-i \e \Phi B \Phi-z)$, $P_1 = i\e \Phi^\bot B \pppg {H_0}^{-\frac 12}$ and $P_2 = \pppg {H_0}^{\frac 12} \Phi$. According to \eqref{estim-GTh-bot} we have 
\[
\e \nr{\pppg {H_0}^{\frac 12} \Phi \GTh \Phi^\bot B \pppg {H_0}^{-\frac 12}}_{\Lc(\Hc)}
\lesssim \e \nr{\Phi \GTh \Phi^\bot}_{\Lc(\Kc^*,\Kc)}\nr{ B}_{\Lc(\Kc,\Kc^*)} \lesssim \sqrt \e.
\]
So if $\e_\bot  \in ]0,\e_\Th]$ is chosen small enough we can apply Lemma \ref{lem-res-perturb} for $\e \in ]0,\e_\bot]$: for all $z \in \C_{I,+}$ and $\e \in ]0,\e_\bot]$ the operator $(\tH - z - i\e B\Phi)$ has a bounded inverse $\Gzpe \in \Lc(\Kc^*,\Kc)$ given by
\[
\Gzpe = \GTh - i \e \GTh   \Phi^\bot B \pppg {H_0}^{-\frac 12} \G^\bot_{z}(\e) \inv   \pppg {H_0}^{\frac 12} \Phi \GTh ,
\]
where 
\[
\G^\bot_{z}(\e) =  1 + i \e \pppg {H_0}^{\frac 12} \Phi \GTh \Phi^\bot B \pppg {H_0}^{-\frac 12}.
\]
Then, as above we use \eqref{estim-GTh}, \eqref{estim-GTh-bot} and \eqref{estim-GTh-Ainv} to prove
\begin{equation} \label{estim-Gzpe}
\nr{\Gzpe} _{\Lc(\Kc^*,\Kc)} \lesssim \frac 1 {\a \e},
\end{equation}
\begin{equation} \label{estim-Gzpe-bot}
\nr{\Gzpe  \Phi^\bot}_{\Lc(\Kc^*,\Kc)} \lesssim \frac 1 {\sqrt{\a} \sqrt \e}
\end{equation}
and
\begin{equation} \label{estim-Gzpe-Ainv}
\nr{\Gzpe  \pppg {A}\inv }_{\Lc(\Hc,\Kc)}  \lesssim \frac 1 { {\a}\sqrt \e}.
\end{equation}

\stepp In order to prove the existence of $\Gzue$, it remains to apply Lemma \ref{lem-res-perturb} with $T = {(\tH- i\e B \Phi -z)}$, $P_1 = i \e B \pppg {H_0}^{-\frac 12}$ and $P_2 = \pppg {H_0}^{\frac 12} \Phi^\bot$. We have
\[
\e \nr{\pppg {H_0}^{\frac 12} \Phi^\bot \GTh  B \pppg {H_0}^{-\frac 12}}_{\Lc(\Hc)}
\lesssim \e \nr{\Phi^\bot \GTh }_{\Lc(\Kc^*,\Kc)}\nr{ B}_{\Lc(\Kc,\Kc^*)} \lesssim \sqrt \e.
\]
So if $\e_1 \in ]0,\e_\bot]$ is chosen small enough we can apply Lemma \ref{lem-res-perturb}, which proves that for $z \in \C_{I,+}$ and $\e \in ]0,\e_1]$ the operator $(\tH -i\e B - z)$ has a bounded inverse $\Gzue \in \Lc(\Kc^*,\Kc)$ given by
\[
\Gzue = \Gzpe - i \e \Gzpe   B \pppg {H_0}^{-\frac 12} \G^1_{z}(\e) \inv   \pppg {H_0}^{\frac 12} \Phi^\bot \Gzpe ,
\]
where 
\[
\G^1_{z}(\e) =  1 + i \e \pppg {H_0}^{\frac 12} \Phi^\bot \GTh B \pppg {H_0}^{-\frac 12}.
\]
Moreover we have
\begin{equation} \label{estim-Gzue}
\nr{\Gzue} _{\Lc(\Kc^*,\Kc)} \lesssim \frac 1 {\a \e}
\end{equation}
and
\begin{equation} \label{estim-Gzue-Ainv}
\nr{\Gzue  \pppg {A}\inv }_{\Lc(\Hc,\Kc)}  \lesssim \frac 1 { {\a}\sqrt \e}.
\end{equation}

\stepp For $n\in\Ii 2 N$ we have 
\begin{align*}
\nr{\pppg {H_0}^{\frac 12} \Gzue (C_n(\e) - C_1(\e))\pppg {H_0}^{-\frac 12} }_{\Lc(\Hc)}
& \leq \sum_{j=2}^n \e^j \nr{\Gzue}_{\Lc(\Kc^*,\Kc)} \nr{B_j}_{\Lc(\Kc,\Kc^*)}\\
& \lesssim \e^2 \times \frac 1 {\a \e} \times \a \lesssim \e.
\end{align*}
Thus for $\e \in ]0,\e_N]$, $\e_N$ chosen small enough, we can apply Lemma \ref{lem-res-perturb} with $T = \tH + C_1(\e) - z$, $P_1 = \big(C_n(\e) - C_1(\e)\big) \pppg {H_0}^{-\frac 12}$ and $P_2 = \pppg {H_0}^{\frac 12}$. This proves that the operator $\tH  + C_n(\e) - z$ has a bounded inverse in $\Lc(\Kc^*,\Kc)$, given by 
\begin{align*}
\Gzne
= \Gzue - \Gzue \big( C_n(\e) - C_1(\e) \big) \pppg {H_0}^{-\frac 12}  \big( 1+ \Gzue (C_n(\e)-C_1(\e))\big) \inv \pppg {H_0}^{\frac 12}  \Gzue.
\end{align*}
This proves the first statement, and the estimates are proved as above. 

\stepp Let $\e \in ]0,\e_N[$. For $\tilde \e \in \big]\frac \e 2,\e_N[$ we have 
\[
\Gzn(\tilde \e) - \Gzn(\e) = - \Gzn(\tilde \e) \big( C_n(\tilde \e) - C_n(\e) \big) \Gzn(\e).
\]
Since $C_n$ is a continuous function in $\Lc(\Kc,\Kc^*)$ and $\Gzn$ is uniformly bounded in $\Lc(\Kc^*,\Kc)$ (by a constant which depends on $\a$) on $\big] \frac \e 2,\e_N \big[$, the map $\Gzn$ is continuous in $\Lc(\Kc^*, \Kc)$. Then we divide this equality by $\tilde \e - \e$ et let $\tilde \e$ go to $\e$. We obtain that $\Gzn$ is differentiable and 
\[
\frac d {d\e} \Gzne = - \Gzne C_n'(\e) \Gzne.
\]
The derivative $C_n'(\e)$ is well-defined in $\Lc(\Kc,\Kc^*)$. In the sense of forms on $\Ec$ we can check that
\[
C_n'(\e) = [\tH + C_n(\e) - z ,A] - \frac {(-i\e)^n}{n!} [B_n,A].
\]
But the right-hand side extends to an operator in $\Lc(\Kc,\Kc^*)$, and the last statement of the proposition follows.
\end{proof}

 The following two results generalize Theorems 3.2 and 3.5 in \cite{jensen85}:

\begin{theorem} \label{thjen3.2}
Suppose $A$ is a conjugate operator for $H$ up to order $N$ on $J$ with bounds $(\a,\b,\Upsilon_N)$. Let $\d_1,\d_2 \geq 0$ be such that $\d_1 + \d_2< N-1$. Let $I$ be a compact subinterval of $\mathring J$. Then there exists $c \geq 0$ which only depends on $C_\Th$, $J$, $I$, $\d_1$, $\d_2$, $\b$ and $\Upsilon_N$ such that for all $z \in \C_{I,+}$ we have
\begin{equation*} 
\nr{\pppg{A}^{\d_1} \1{\R_-} (A) (H - z)\inv \1{\R_+}(A) \pppg{A}^{\d_2}} \leq \frac c {\a}.
\end{equation*}
Moreover for $\Re(z) \in \mathring J$ fixed this operator has a limit when $\Im(z) \searrow 0$. This limit defines in $\Lc(\Hc)$ a H\"older-continuous function of index $\frac {N-1-\d_1-\d_2}{N+1}$ with respect to $\Re(z)$.
\end{theorem}

\begin{proof}
Let $\e_N$ be given by Proposition \ref{prop-gnze}. For $z \in \C_{I,+}$ and $\e \in ]0,\e_N]$ we set
\[
F_z^N(\e) = \pppg{A}^{\d_1}e^{\e A}\1{\R_-} (A) G^N_z (\e)  \1{\R_+}(A) e^{- \e A} \pppg{A}^{\d_2}.
\]
According to Proposition \ref{prop-gnze}, the functional calculus and the fact that $\nr{[B_N,A]}_{\Lc(\Kc,\Kc^*)} \lesssim \a$ we have
\[
\begin{aligned}
\nr{\frac  d {d\e} F_z^N(\e)}
& = \frac {\e^N}{N!} \nr{\pppg{A}^{\d_1} e^{\e A} \1{\R_-}(A)  \GzNe [B_N,A] \GzNe  \1{\R_+}(A) e^{- \e A} \pppg{A}^{\d_2}}\\
& \lesssim \e^{-\d_1} \times  \a\inv \e^{N-2}\times \e ^{-\d_2} = \frac {\e ^{N - \d_1 - \d_2 -2}}\a.
\end{aligned}
\]
Since $N - \d_1 - \d_2 -2 > -1$, this proves that $F_z^N(\e)$ is uniformly bounded (we do not have to use Lemma \ref{lem-JMP} here).
Now let $z,z' \in \C_{I,+}$ and $\e\in]0,\e_0]$. The previous estimates give
\[
\nr{F^N_z(\e) - F^N_{z}(0)} \leq \frac c {\a} \e ^{N-1 - \d_1-\d_2} \quad \text{and} \quad \nr{F^N_z(\e) - F^N_{z'}(\e)} \leq \frac c {\a^2} \e ^{-(\d_1+\d_2+2)} \abs{z-z'}.
\]
We get the second statement as we did for Theorem \ref{th-mourre-form}, taking $\e = \a ^{-\frac 1{N+1}} \abs{z-z'}^{\frac 1 {N+1}}$.
\end{proof}

\begin{theorem} \label{thjen3.5}
Suppose $A$ is a conjugate operator for $H$ up to order $N$ on $J$ with bounds $(\a,\b,\Upsilon_N)$. Let $\d \in \left] \frac 12 , N \right[$. Then there exists $c \geq 0$ which only depends on $C_\Th$, $J$, $I$, $\d_1$, $\d_2$, $\b$ and $\Upsilon_N$ such that for all $z \in \C_{I,+}$ we have
\begin{equation*} 
\nr{\pppg{A}^{-\d} (H - z)\inv \1{\R_+}(A) \pppg{A}^{\d-1}}_{\Lc(\Hc)}  \leq \frac c {\a},
\end{equation*}
and for $\Re(z) \in \mathring J$ fixed this operator has a limit when $\Im(z) \searrow 0$. This limit defines in $\Lc(\Hc)$ a H\"older-continuous function with respect to $\Re(z)$. Moreover we have similar results for the operator
\begin{equation*} 
 \pppg{A} ^{\d-1}\1 {\R_-}(A)(H -  z) \inv \pppg{A}^{-\d}.
\end{equation*}
\end{theorem}

\begin{proof} 
We follow the proof given in \cite{jensen85}. It relies itself on the results of \cite{mourre83}. We also refer to \cite{these} for a proof in the dissipative case (perturbation by a dissipative operator). The case of a dissipative perturbation in the sense of forms does not rise new difficulties, so we omit the details.
\end{proof}

Now that we have Theorems \ref{th-mourre-form}, \ref{thjen3.2} and \ref{thjen3.5} we can follow the idea developped in \cite[Sec. 5]{boucletr14}. The purpose is not only to prove uniform estimates for the powers of the resolvent, but also to allow inserted factors. This is motivated by the wave equation. Indeed, the derivatives of the corresponding resolvent are not its powers in this case (see Example \ref{ex-wave} below).\\

Let $n \in \Ii 1 N$. We consider $\Phi_0 \in \Lc(\Kc,\Hc)$, $\Phi_1,\dots,\Phi_{n-1} \in \Lc(\Kc,\Kc^*)$ and $\Phi_n \in \Lc(\Hc,\Kc^*)$. We assume (inductively) on $m \in \Ii 1 N$ that the operator 
\[
\ad_{iA}^m (\Phi_0) := [\ad_{iA}^{m-1} (\Phi_0),iA]
\]
(with $\ad_{iA}^0 (\Phi_0) = \Phi_0$), at least defined as an operator in $\Lc(\Ec,\Ec^*)$, can be extended to an operator in $\Lc(\Kc,\Hc)$. We assume similarly that the commutators $\ad_{iA}^m (\Phi_j)$ for $m \in \Ii 1 N$ and $j \in \Ii 1 {n-1}$ extend to operators in $\Lc(\Kc,\Kc^*)$ and finally that the commutators $\ad_{iA}^m (\Phi_n)$ for $m \in \Ii 1 N$ extend to operators in $\Lc(\Hc,\Kc^*)$. Then for $j \in \Ii 1 {n-1}$ we set
\[
\nr{\Phi_j}_{\Cc_N(A,\Kc,\Kc^*)} = \sum_{m=0}^N \nr{\ad_{iA}^m (\Phi_j)}_{\Lc(\Kc,\Kc^*)}.
\]
We similarly define $\nr{\Phi_0}_{\Cc_N(A,\Kc,\Hc)}$ and $\nr{\Phi_n}_{\Cc_N(A,\Hc,\Kc^*)}$, and then
\[
\nr{(\Phi_0,\dots,\Phi_n)}_{\Cc_N^n} = \nr{\Phi_0}_{\Cc_N(A,\Kc,\Hc)} \nr{\Phi_n}_{\Cc_N(A,\Hc,\Kc^*)}\prod_{j=1}^{n-1} \nr{\Phi_j}_{\Cc_N(A,\Kc,\Kc^*)}.
\]
For $z \in \C_+$ we set  
\begin{equation} \label{def-Rc}
\Rc_n(z) = \Phi_0 (H-z)\inv  \Phi_1 (H-z)\inv \dots \Phi_{n-1} (H-z)\inv \Phi_n.
\end{equation}
The statement is the following:

\begin{theorem} \label{th-mourre-insert}
Suppose that the self-adjoint operator $A$ is conjugate to the maximal dissipative operator $H$ on $J$ up to order $N$ with bounds $(\a,\b,\Upsilon_N)$. Let $I \subset \mathring J$ be a compact interval. Let $\d \in \big] n- \frac 12,N\big[$ and $\d_1,\d_2 \geq 0$ such that $\d_1 + \d_2 < N -n$. Then there exists $c \geq 0$ such that
\begin{equation*}
\nr{ \pppg{A}^{-\d} \Rc_n(z) \pppg{A_\l}^{-\d}} \leq \frac c {\a^n}  \nr{(\Phi_0,\dots,\Phi_n)}_{\Cc_N^n},
\end{equation*}
\begin{equation*}
\nr{ \pppg{A}^{\d-n} \1{\R_-}(A) \Rc_n(z)   \pppg{A}^{-\d}} \leq \frac c {\a^n} \nr{(\Phi_0,\dots,\Phi_n)}_{\Cc_N^n},
\end{equation*}
\begin{equation*}
\nr{ \pppg{A}^{-\d} \Rc_n(z)  \1{\R_+}(A) \pppg{A}^{\d-n}} \leq \frac c {\a^n} \nr{(\Phi_0,\dots,\Phi_n)}_{\Cc_N^n}
\end{equation*}
and
\begin{equation*}
\nr{ \pppg{A}^{\d_1} \1{\R_-} (A) \Rc_n(z)  \1{\R_+} (A) \pppg{A}^{\d_2}} \leq \frac c {\a^n}  \nr{(\Phi_0,\dots,\Phi_n)}_{\Cc_N^n}.
\end{equation*}
\end{theorem}

\begin{proof}
We can follow the proof of the analogous Theorem 5.14 in \cite{boucletr14}. We only briefly recall the strategy. With the identity
\[
(H-z)\inv = (H-i)\inv + (z-i) (H-i)^{-2} + (z-i)^2 (H-i)\inv (H-z)\inv (H-i)\inv,
\]
we see that we can assume without loss of generality that the operators $\Phi_j$ and their commutators with $A$ are in $\Lc(\Hc)$. Then the idea is to start from the estimates for a single resolvent (see Theorems \ref{th-mourre-form}, \ref{thjen3.2} and \ref{thjen3.5}), to prove analog estimates with $(H-z)\inv$ replaced by an operator of the form $\Phi_j (H-z)\inv \Phi_k$ (for this we use the commutation properties between $\Phi_j$ and $A$), and finally we use Lemma 5.4 in \cite{boucletr14} to obtain the multiple resolvent estimates with inserted factors. We omit the details and refer to the proof of Theorem 5.14 in \cite{boucletr14}.
\end{proof}

\begin{remark}
With the same idea we could even prove uniform estimates for an operator of the form 
\[
\Rc(z) = \Phi_0 (H_1-z)\inv  \Phi_1 (H_2-z)\inv \dots \Phi_{n-1} (H_n-z)\inv \Phi_n,
\]
where $H_1,\dots,H_n$ are different maximal dissipative operators of the form dicussed in Section \ref{sec-def} with uniform constant $C_\Th$ in \eqref{qI-qO-bounded} and with the same form domain $\Kc$, under the assumption that $A$ is conjugated to $H_k$ on $J$ with bounds $(\a_k,\b,\Upsilon_N)$ for all $k \in \Ii 1 n$. Then the quotient $\a^n$ is replaced by $\a_1\dots\a_n$ in the estimates of the theorem.
\end{remark}

\begin{example} \label{ex-wave}
We consider the wave equation \eqref{eq-wave} on the half-space \eqref{def-half-space}. Assume that $w_0 = 0$ on $\partial \O$. Let $w$ be the solution of \eqref{eq-wave}. For $\m > 0$ we set $w_\m(t) = \1 {\R_+}(t) e^{-t\m} w(t)$. Then the inverse Fourier transform of $w_\m$,
\[
\check w_\m(\t) = \int_\R e^{it\t} w_\m(t) \, dt = \int_0^{+\infty} e^{it(\t+i\m)} w(t) \, dt,
\]
is solution of the problem 
\[
\begin{cases}
(-\D - z^2) \check w_\m (\t) = -izw_0 + w_1 & \text{on } \O,\\
\partial_\n \check w_\m(\t) = iza \check w_\m(\t) & \text{on } \partial \O,
\end{cases}
\]
where $z = \t +i\m$.  In other words, we have 
\[
\check w_\m(\t) = R(z)( -izw_0 + w_1) \quad \text{where} \quad R(z) =(H_{az} - z^2)\inv  .
\]
In order to study the properties of $\check w_\m(\t)$ and hence those of $w(t)$ we have to prove uniform resolvent estimates for the derivative of $R(z)$ when $\Im (z) \searrow 0$ (see for instance Theorem 1.2 in \cite{boucletr14} for the wave equation on $\R^d$). We can check that for $z \in \C_+$ we have 
\[
R'(z) = R(z) (i \Th + 2z) R(z),
\]
where $\Th \in \Lc (H^1(\O),H^1(\O))$ is the operator corresponding to the imaginary part $q_\Th$ of $q_a$ (see \eqref{def-qa}). Following Proposition 5.9 in \cite{boucletr14} we can check that for $n \in \N^*$ the derivative $R^{(n)}(z)$ is a linear combination of terms of the form 
\[
z^k R(z) \Th^{j_1} R(z) \Th^{j_2} R(z) \dots \Th^{j_m} R(z),
\]
where $m \in \Ii 0 n$ (there are $m+1$ factors $R(z)$), $k\in\N$, $j_1,\dots,j_m \in \{0,1\}$, $\Th^1 = \Th$, $\Th^0 = \Id$ and $n = 2m - k - (j_1 + \dots + j_m)$. The difference is that $\Th$ is not a bounded operator on $L^2$. However, we have checked the commutation properties between $\Th$ and $A$ in the proof of Proposition \ref{prop-guide}, so with Theorem \ref{th-mourre-insert} we can prove the following result:
\end{example}

\begin{proposition}
Let $n\in\N$ and assume that \eqref{hyp-a-A} holds for $N\geq n$. Let $\d > n + \frac 12$ and let $I$ be a compact subset of $\R_+^*$. Then there exists $C \geq 0$ such that for all $z \in \C_{I,+}$ we have 
\[
\nr{\pppg x^{-\d} R^{(n)}(z) \pppg x^{-\d}}_{\Lc(L^2)} \leq C.
\]
\end{proposition}

\section{Absolutely continuous spectrum} \label{sec-Hac}

\newcommand{\dilHc}{{\hat \Hc}}
\newcommand{\dilH}{{\hat H}}
\newcommand{\wh}{{\sqrt a}}
\newcommand{\Ecc}{{L^2(\partial \O)}}

In this section we discuss the properties of the absolutely continuous subspace for a dissipative operator. We recall from \cite{davies78} the following definition:

\begin{definition} \label{def-Hac}
Let $H$ be a maximal dissipative operator on a Hilbert space $\Hc$. The absolutely continuous subspace $\Hc_{ac}(H)$ of $H$ is the closure in $\Hc$ of 
\[
\Hc_{ac}^*(H) := \singl{ \f \in \Hc \st \exists C_\f \geq 0, \forall \p \in \Hc,  \int_0^{+\infty} \abs {\innp{ e^{-itH} \f}\p_\Hc}^2 \, dt \leq C_\f \nr \p_\Hc^2}. 
\]
\end{definition}

For a self-adjoint operator this definition coincide with the usual definition involving the spectral measure (see for instance Proposition 1.7, Theorem 1.3 and Corollary 1.4 in \cite{perry}).\\

In the self-adjoint case, the uniform resolvent estimates and the $L^2(\R_+,\Hc)$ norm of the solution of the time-dependant problem are linked by the theory of relatively smooth operators in the sense of Kato (see \cite{kato66} and \cite[\S XIII.7]{rs4}). It is less known that this link remains valid for dissipative operators.

In order to extend the self-adjoint theory of relative smoothness for a dissipative operator $H$, we use a self-ajdoint dilation of $H$. For the general theory of self-adjoint dilations we refer to \cite{nagyf}. Here we only recall that a maximal dissipative operator $H$ on a Hilbert space $\Hc$ always has a self-adjoint dilation. This means that there exists a self-adjoint operator $\dilH$ on some Hilbert space $\dilHc$ (which contains $\Hc$ as a subspace) such that on $\Lc(\Hc)$ we have 
\begin{eqnarray*}
\forall z \in \C_+ ,&& P_\Hc (\dilH-z)\inv I_\Hc = (H-z)\inv, \\ 
\forall z \in \C_+ ,&& P_\Hc (\dilH-\bar z)\inv I_\Hc = (H^*-\bar z)\inv ,\\ 
\forall t \geq 0 , && P_\Hc e^{-it \dilH} I_\Hc = e^{-itH}, \\ 
\forall t \geq 0 , && P_\Hc e^{it \dilH} I_\Hc = e^{itH^*}, 
\end{eqnarray*}
where $P_\Hc \in \Lc(\dilHc,\Hc)$ denotes the orthogonal projection of $\dilHc$ on $\Hc$ and $I_\Hc \in \Lc(\Hc,\dilHc)$ is the embedding of $\Hc$ in $\dilHc$. An explicit example of (minimal) self-adjoint dilation for the dissipative Schr\"odinger operator on $\R^d$ is given in \cite{pavlov77}.

\begin{proposition} \label{prop-Hac}
Let $Q$ be a closed operator on $\Hc$. Assume that there exists $C \geq 0$ such that for all $z \in \C_+$ and $\f \in \Dom(Q^*)$ we have 
\[
\innp{ \big((H-z)\inv - (H^* - \bar z)\inv \big) Q^* \f}{Q^* \f} _\Hc \leq C \nr{\f}^2_\Hc .
\]
Then for $\p \in \Hc$ we have $e^{-itH}\p \in \Dom(Q)$ for almost all $t \geq 0$ and 
\[
\int_0^{+\infty} \nr{Q e^{-itH} \p}_\Hc^2 dt \leq C \nr \p ^2_\Hc.
\]
We also have $e^{itH^*}\p \in \Dom(Q)$ for almost all $t \geq 0$ and 
\[
\int_0^{+\infty} \nr{Q e^{itH^*} \p}_\Hc^2 dt \leq C \nr \p ^2_\Hc.
\]
\end{proposition}

\begin{proof} 
Let $\dilH$ be a self-adjoint dilation of $H$ on a Hilbert space $\dilHc$ which contains $\Hc$ as a subspace. We can write $\dilHc = \Hc \oplus \Hc^\bot$. We extend $Q$ as an operator $\hat Q$ on $\dilHc$ by 0 on $\Hc^\bot$. Then $\hat Q$ is a closed operator on $\dilHc$ with domain $\Dom(\hat Q) = \Dom(Q) \oplus \Hc^\bot$. Then for all $z \in \C_{+}$ and $\hat \f = (\f,\f^\bot), \hat \p = (\p ,\p^\bot) \in \Dom(\hat Q)$ we have
\begin{align*}
\innp {\big( (\dilH-z)\inv - (\dilH - \bar z) \inv \big) \hat Q^* \hat \f} {\hat Q^* \hat \p} _\dilHc
& = \innp {\big( (H-z)\inv - (H^* - \bar z) \inv \big)  Q^*  \f} { Q^*  \p} _\Hc\\
& \leq C \nr {\f}_\Hc \nr { \p}_\Hc  \leq C \nr {\hat \f}_\dilHc \nr {\hat \p}_\dilHc.
\end{align*}
Let $\hat \z \in \dilHc$. According to Theorem XIII.25 in \cite{rs4} we have $e^{-it\dilH}\hat \z \in \Dom(\hat Q)$ for almost all $t \in \R$ and
\[
\int_\R \nr{\hat Q e^{-it\dilH} \hat \z}^2_{\dilHc} \, dt \leq C\Vert \hat \z \Vert^2_{\dilHc}.
\]
Now let $\f \in \Hc$ and $\hat \f = (\f,0) \in \dilHc$. We have $e^{-itH} \f = P_\Hc e^{-it\dilH}\hat \f \in P_\Hc \Dom(\hat Q) = \Dom(Q)$ for almost all $t \geq 0$ and moreover
\[
\int_0^{+\infty} \nr{Q e^{-itH} \f}^2_\Hc \, dt = \int_0^{+\infty} \nr{\hat Q e^{-it\dilH} \hat\f}^2_\dilHc \, dt \leq C \Vert \hat \f\Vert^2_{\dilHc} = C \nr\f_\Hc^2.
\]
We conclude similarly for the integral of $\nr{Q e^{itH^*}\f}^2_\Hc$.
\end{proof}

\begin{corollary} \label{cor-Hac}
Under the assumptions of Proposition \ref{prop-Hac} we have $\Ran(Q^*) \subset \Hc_{ac}^* (H)$.
\end{corollary}

\begin{proof}
Let $\f \in \Ran(Q^*)$ and $\z \in \Hc$ be such that $\f = Q^* \z$. Then for $\p \in \Hc$ we have 
\[
\int_0^{+\infty} \abs{\innp{e^{-itH}\f}{\p}_\Hc}^2\, dt \leq \int_{0}^{+\infty} \nr{\z}_\Hc^2 \nr{Q e^{itH^*} \p}_\Hc^2 \,dt \leq C \nr{\z}_\Hc^2 \nr{\p}_\Hc^2. \qedhere
\]
\end{proof}

Theorem \ref{th-mourre-form} gives an estimate as in Proposition \ref{prop-Hac} with $Q = \pppg {A}^{-\d}$ but only for $z \in \C_{I,+}$ for some interval $I$. In order to obtain an estimate for all $z \in \C_+$ we have to localize spectrally. For this we are going to use a function of the self-adjoint part $H_0$ of $H$. We first prove the following lemma:

\begin{lemma} \label{lem-comm-AchiH}
Let $N\in\N^*$. Similarly to Definition \ref{def-conj-N}, assume inductively that the commutators $B_n^0 := \ad^n_{iA}(\tilde H_0)$, at least defined as operators in $\Lc(\Ec,\Ec^*)$, extend to bounded operators on $\Lc(\Kc,\Kc^*)$ for $n=1,\dots,N$. Then for all $\d \in [-N,N]$ and $\h \in C_0^\infty(\R)$ the operator $\pppg A^{-\d} \h(H_0) \pppg A^\d$ extends to a bounded operator in $\Lc(\Kc^*,\Kc)$.
\end{lemma}

\begin{proof}
We consider an almost analytic extension $\tilde \h$ of $\h$ (see \cite{dimassis, davies95}):
\begin{equation*} 
 \tilde \h (x+iy) = \p(y) \sum_{k=0}^m \h^{(k)} (x) \frac {(iy)^k}{k!}
\end{equation*}
where $m \geq N+1$ and $\p \in C_0^\infty(\R,[0,1])$ is supported on $[-2,2]$ and equal to 1 on $[-1,1]$. We have
\begin{align*}
 \frac {\partial \tilde \h} {\partial \bar \z} (x+iy) = \frac {i \p'(y)} 2  \sum_{k=0}^m \h^{(k)} (x) \frac {(iy)^k}{k!} + \frac {\p(y)} 2  \h^{(m+1)} (x) \frac {(iy)^m}{m!} ,
\end{align*}
and in particular for $\z \in \C$
\begin{equation} \label{estim-tildef}
\begin{aligned}
\abs{\frac {\partial \tilde \h}{\partial \bar \z} (\z)} \lesssim   \abs{\Im \z}^m  \1 { \singl{\Re(\z) \in \supp(\h) , \abs{\Im(\z)} \leq 2}}(\z).
\end{aligned}
\end{equation}
Thus we can write the Helffer-Sj\"ostrand formula for $\h(H_0)$:
\begin{equation*}
\h(H_0) = - \frac 1 \pi \int_{\z = x + i y \in \C} \frac {\partial \tilde \h}{\partial \bar \z} (\z) (\tilde H_0-\z)\inv \, dx \, dy.
\end{equation*}
Then we can check by induction on $n\in\Ii 1 N$ that the commutator $\ad_{iA}^n(\h(H_0))$ can be written as a sum of terms of the form 
\[
\pm \frac 1 \pi \int_{\C} \frac {\partial \tilde \h}{\partial \bar \z} (\z) (\tilde H_0-\z)\inv B_{n_1}^0 (\tilde H_0-\z)\inv B_{n_2}^0 \dots(\tilde H_0-\z)\inv B_{n_p}^0 (\tilde H_0-\z)\inv \, dx \, dy
\]
where $p \in \Ii 1 {n}$ and $n_1,\dots,n_p \in \N^*$ are such that $n_1 + \dots + n_p = n$. For $\z \in \supp(\tilde \h)$ and $\Im(\z) \neq 0$ we have 
\[
\nr{(\tilde H_0-\z)\inv}_{\Lc(\Kc^*,\Kc)} \lesssim \frac 1 {\Im(\z)}
\]
so with \eqref{estim-tildef} we see that $\ad_{iA}^n(\h(H_0))$ extend to an operator in $\Lc(\Kc^*,\Kc)$ for $n \in \Ii 0 N$. Thus $\pppg A^{-\d} \h(H_0) \pppg A^\d$ extends to a bounded operator for $\d \in \Ii {-N} N$, and we conclude by interpolation.
\end{proof}

Now we can prove the main result of this section:

\begin{theorem} \label{th-abs-spectrum}
Assume that $A$ is a conjugate operator to $H$ on the open interval $J$, and let $\d > \frac 12$. Then $\Ran \big(\1 J (H_0)\pppg A^{-\d} \big) \subset \Hc_{ac}(H)$ .
\end{theorem}

\begin{proof} 
Let $I$ and $I'$ be compact intervals such that $I \subset \mathring I' \subset I' \subset J$. Let $\h \in C_0^\infty(\R)$ be supported in $\mathring I$ and equal to 1 on a neighborhood of $I$. According to Lemma \ref{lem-comm-AchiH} and Theorem \ref{th-mourre-form} (and Remark \ref{rem-adjoint}) the operator
\[
\pppg A^{-\d} \h(H_0) (H-z)\inv \h(H_0) \pppg A^{-\d}
\]
and its adjoint are bounded in $\Lc(\Hc)$ uniformly in $z \in \C_{I',+}$. Then for $z \in \C_{\R \setminus I', +}$ and $\f,\p \in \Hc$ we have by the resolvent identity (see \eqref{eq-res-identity}):
\begin{eqnarray*}
\lefteqn{\abs{\innp{\pppg A^{-\d} \h(H_0) (H-z)\inv \h(H_0) \pppg A^{-\d} \f} \p _\Hc}}\\
&& \leq \abs{\innp{\pppg A^{-\d} \h(H_0) (H_0-z)\inv \h(H_0) \pppg A^{-\d} \f} \p _\Hc}\\
&& \quad  +  q_\Th \big( (H-z)\inv \h(H_0) \pppg A^{-\d}\f,  (H_0- \bar z)\inv \h(H_0) \pppg A^{-\d}\p\big)\\
&& \lesssim \nr\f \nr \p + q_\Th \big( (H-z)\inv \h(H_0) \pppg A^{-\d}\f \big)^{\frac 12} q_\Th \big(  (H_0- \bar z)\inv \h(H_0) \pppg A^{-\d}\p \big)^{\frac 12} 
\end{eqnarray*}
According to Proposition \ref{prop-estim-quad-form} and \eqref{qI-qO-bounded} we have
\[
\nr{\pppg A^{-\d} \h(H_0) (H-z)\inv \h(H_0) \pppg A^{-\d}} \lesssim  1 + \nr{\pppg A^{-\d} \h(H_0) (H-z)\inv \h(H_0) \pppg A^{-\d}}^{\frac 12} ,
\]
and hence
\[
\nr{\pppg A^{-\d} \h(H_0) (H-z)\inv \h(H_0) \pppg A^{-\d}} \lesssim 1. 
\]
Since we have the same estimate for $(H^*-\bar z)\inv$ instead of $(H-z)\inv$, we conclude with Corollary \ref{cor-Hac} that $\Ran(\h(H_0) \pppg A^{-\d}) \subset \Hc_{ac}^*(H)$.  Since $\Hc_{ac}(H_a)$ is closed in $\Hc$ by definition, the result follows.
\end{proof}

We go back to the Schr\"odinger operator on the dissipative wave guide discussed in Section \ref{sec-guide}, see \eqref{def-Ha}-\eqref{def-Dom-Ha}. In \cite{art-diss-schrodinger-guide} we have proved that under a strong assumption on the absorption index then the norm of the solution of the Schr\"odinger equation \eqref{eq-schrodinger} decays exponentially, which implies in particular that $\Hc_{ac}(H_a) = L^2(\O)$. In general we cannot expect such a fast decay but the result concerning the $L^2_t \big(\R_+,L^2(\O)\big)$ norm remains valid:

\begin{proposition}
With the notation of Section \ref{sec-guide} we have $\Hc_{ac}(H_a) = L^2(\O)$.
\end{proposition}

\begin{proof}
We recall that $\Tc$ is the (discrete) set of thresholds. Let $\d > \frac 12$. If $J \Subset \R \setminus \Tc$ then according to Theorem \ref{th-abs-spectrum} we have $\Ran \big(\1 J (H_0) \pppg {A_x}^{-\d}\big) \subset \Hc_{ac}(H_a)$. Since $H_0$ has no eigenvalue, the union of these sets for all suitable $J$ is dense in $L^2(\O)$. Since $\Hc_{ac}(H_a)$ is closed in $L^2(\O)$, the result follows.
\end{proof}

\bibliographystyle{alpha}
\bibliography{bibliotex}

\end{document}